\documentclass[reqno]{amsart}

\usepackage[foot]{amsaddr}

\usepackage{amsmath,amssymb,amsfonts,mathrsfs}
\usepackage{comment}
\usepackage{verbatim}
\usepackage{bbold}
\usepackage{tikz-cd}
\usepackage[english]{babel}
\usepackage[utf8]{inputenc}
\usepackage[T1]{fontenc}
\usepackage{mathtools}
\usepackage{lmodern}
\usepackage{comment}
\usepackage{cancel}

\usepackage[
top=3cm,
bottom=3cm,
left=3cm,
right=3cm,
marginparwidth=2.5cm
]{geometry}

\usepackage[colorinlistoftodos]{todonotes}
\usepackage[hypertexnames=false]{hyperref}
\hypersetup{
    colorlinks,
    linkcolor={red!50!black},
    citecolor={blue!50!black},
    urlcolor={blue!80!black},
}

\usepackage[shortlabels]{enumitem}
\setlist[1]{itemsep=2pt}
\usepackage[capitalise]{cleveref}

\usepackage{autonum} %to number only \eqref'ed equations in a cref compatible way

\newtheorem{theorem}{Theorem}[section]
\newtheorem{prop}[theorem]{Proposition}
\newtheorem{corollary}[theorem]{Corollary}
\newtheorem{lemma}[theorem]{Lemma}
\newtheorem{defi}[theorem]{Definition}

\theoremstyle{remark} 
\newtheorem{remark}[theorem]{Remark}

\newcommand{\R}{\mathbb{R}}
\newcommand{\N}{\mathbb{N}}
\newcommand{\ve}{\varepsilon}

\DeclareMathOperator{\rank}{\mathrm{rank}}

\title[]{Curvature measures and the sub-Riemannian\\ Gauss-Bonnet theorem}
\author[D. Barilari]{Davide Barilari}

\author[E. Bellini]{Eugenio Bellini}

\address[Davide Barilari, Eugenio Bellini]{Dipartimento di Matematica "Tullio Levi-Civita", Universit\`a degli Studi di Padova, via Trieste 63, 35131 Padova (PD), Italy\\
\textbf{\emph{davide.barilari@unipd.it,\, eugenio.bellini@unipd.it}}}
\author[A.~Pinamonti]{Andrea Pinamonti}
\address[Andrea Pinamonti]{Dipartimento di Matematica, Universit\`a di Trento,
Via Sommarive, 14, 38123 Povo TN, Italy\\
\textbf{\emph{andrea.pinamonti@unitn.it}}}

\begin{document}
\maketitle
\begin{abstract}
We adopt a measure-theoretic perspective on the Riemannian approximation scheme proving a sub-Riemannian Gauss-Bonnet theorem for surfaces in 3D contact manifolds. We show that the zero-order term in the limit is a singular measure supported on isolated characteristic points. In particular, this provides a unified interpretation of previous results obtained in \cite{Balogh1,grong2024}.

Moreover we give natural geometric conditions under which our result holds, namely when the surface admits characteristic points of finite order of degeneracy. This notion, which we introduce, extends the concept of mildly degenerate characteristic points of \cite{Rossi_deg} for the Heisenberg group.  

As a byproduct, we prove that the mean curvature around an isolated characteristic point of finite order of degeneracy is locally integrable.  In particular, this positively answers a question by \cite{DNG} for analytic surfaces in every analytic 3D contact manifold.
\end{abstract}
\tableofcontents

\section{Introduction}

Embedded surfaces play a fundamental role in the theory of contact structures. Since the seminal works of Giroux~\cite{giroux1, giroux2}, they have become central tools in the classification of contact structures, marking a turning point in the study of 3-dimensional contact topology. From a different perspective, works such as~\cite{Blair, Massot} have explored the interplay between contact and Riemannian geometry, introducing analytical techniques that have deepened our understanding of the geometry of contact manifolds.

The study of surfaces in contact 3-manifolds has also begun to attract attention from the point of view of sub-Riemannian geometry. Contributions such as~\cite{Daniele,BellBosc,karen} have addressed the existence and properties of induced metric structures and the corresponding stochastic geometric evolutions, while other works have explored the tightness of distributions in terms of sub-Riemannian curvature invariants~\cite{ABBR2024}.

 Within the sub-Riemannian framework, one of the main challenges is the establishment of a canonical notion of Gaussian curvature and of a Gauss-Bonnet type theorem.
Recent contributions which are close in spirit with our investigation were achieved in \cite{Balogh1,grong2024, DV16,V23} (cf.\ also \cite{ABS08} for the almost-Riemannian setting). In particular, the results in \cite{Balogh1,grong2024} rely on the so-called \emph{Riemannian approximation scheme}, i.e., a family of Riemannian metrics $g^{\ve}$ converging to the sub-Riemannian one. This is a powerful tool to transport results from classical Riemannian geometry into the sub-Riemannian context.

Our goal in this paper is to continue the above-mentioned investigation by developing a \emph{measure-theoretic viewpoint} on the Gauss--Bonnet theorem for  surfaces embedded in $3$-dimensional contact sub-Riemannian manifolds, thereby providing a unified perspective on the previous results.  

 The main novelty of our approach is that, instead of focusing on the limit of the Gaussian curvature functions $K^{\ve}$ associated with the Riemannian approximation scheme $g^{\ve}$, we directly look at the limit of the Gaussian \emph{curvature measures} $K^{\ve}\sigma^{\ve}$, where $\sigma^{\ve}$ is the corresponding  surface area form. In doing so, we are able to characterize both an absolutely continuous and a singular part of this limit, the singular part being concentrated on characteristic points. The notion of a curvature measure, as employed here, is inspired by the seminal work of Federer \cite{F59}. See  also \cite{BBR26} for a similar approach on tube formula in the sub-Riemannian setting.
 
\subsection{Characteristic vector field and its singularities}
 In order to state our results, we review some preliminaries on contact sub-Riemannian manifolds. More details are given in Sections~\ref{sec:preliminaries}, \ref{sec:surfaces_in_3D_cont}.
 
  A three-dimensional contact sub-Riemannian manifold is a triple $(M,\mathcal{D},g)$, where $M$ is a smooth three-dimensional manifold, $\mathcal{D}$ is a \emph{contact distribution} and $g$ a \emph{sub-Riemannian metric}, i.e., a smooth metric on $\mathcal D$. Given an oriented surface $S$ embedded in $M$, the intersection $TS\cap\mathcal D$ defines a singular field of directions, whose singularities are the points $p\in S$ where $T_pS=\mathcal D_p$, called characteristic points. The characteristic set $\Sigma(S)$ is the set of characteristic points of $S$. 

The sub-Riemannian metric together with the orientation of $S$ induce a choice of a unique  \emph{characteristic vector field}, which is a field $X\in \mathfrak{X}(S)$ satisfying 
        \begin{equation}
            \mathrm{span}_\mathbb R\{X_p\}=\begin{cases}
                T_pS\cap \mathcal D_p, &p\notin \Sigma(S),\\
                0,\qquad \qquad \quad  &p\in \Sigma(S).
            \end{cases}
        \end{equation}

 We recall that given a smooth manifold $S$ and a smooth vector field $X\in\mathfrak X(S)$ vanishing at $q\in S$, the differential of $X$ at $q$ is the linear map $D_qX:T_qS\to T_qS$ defined by $D_qX(Y)=[Y,X]_q$.

We say that a characteristic point is \emph{degenerate} if $\det(D_{q}X)= 0$, and \emph{non-degenerate} otherwise.

\begin{remark}\label{r:ker1}
One can prove that, since $\mathcal D$ a contact distribution, we have that $\mathrm{div}(X)|_{q}\neq 0$. The latter condition is independent of the measure used to compute the divergence. Being $\mathrm{trace}(D_{q}X)=\mathrm{div}(X)|_{q}$, at degenerate characteristic points one necessarily has $\dim \ker (D_{q}X)=1$. 

For more details we refer to Section~\ref{sec:surfaces_in_3D_cont} (cf.\ in particular  Remark~\ref{rmk:div_non_zero}). 

\end{remark}

\subsection{The Riemannian approximation}
It is well-known \cite{Agrachev} that the sub-Riemannian metric  $g$ induces a canonical choice of a vector field $f_0\in \mathfrak{X}(M)$ called the Reeb vector field, which is globally transverse to the distribution $\mathcal{D}$ (see \cref{sec:preliminaries} for a more detailed definition). We consider the family of Riemannian metrics $g^\varepsilon$ approximating $g$ for which
\begin{equation}\label{eq:g_ve_properties0}
                g^\ve|_{\mathcal{D}}=g,\qquad g^\ve(\mathcal{D},f_0)=0,\qquad g^\ve(f_0,f_0)=\frac{1}{\ve}.
            \end{equation}
Given a compact orientable surface $S$ embedded in $M$, we denote with $K^\ve$ and $\sigma^\ve$ the Gaussian curvature and area form of the Riemannian surface $(S,g^\ve|_S)$, respectively. Our approach to achieve a Gauss-Bonnet theorem consists in studying the family of signed Borel measures $K^\ve\sigma^{\ve}$ as $\ve\to 0$. Unfortunately, the latter sequence does not converge as $\ve\to 0$. However, the rescaled sequence 
\[
\sqrt{\varepsilon}\, K^\varepsilon \sigma^\varepsilon
\]
strongly converges in $C^{1}(S)^{*}$, as $\varepsilon \to 0$, to a signed Borel measure, which we denote $\mu_{-1}$, with vanishing total mass. This limit measure is absolutely continuous with respect to the smooth area measure $\sigma^1$, although its density may be unbounded near characteristic
	points. Nonetheless, the existence of the limit measure $\mu_{-1}$ allows us to identify and remove the divergent term from $ K^\varepsilon \sigma^\varepsilon$: in our main result we prove that the difference
\[
 K^\varepsilon \sigma^\varepsilon - \frac{1}{\sqrt{\varepsilon}} \mu_{-1}
\]
 strongly converges in $C^{1}(S)^{*}$ to a sum of Dirac deltas supported at the characteristic points, each weighted by the \emph{topological index} of the characteristic vector field. 

Our main result can be stated under the assumption that the characteristic points are \emph{of finite order of degeneracy}. This means that either the characteristic points are non-degenerate, or, if degenerate, their order of degeneracy is finite along a (suitably defined) curve tangent to the kernel of the differential. The notion introduced here generalizes the concept of mildly degenerate characteristic points, as defined in \cite{Rossi_deg} for the Heisenberg group. We refer to Section~\ref{ss:finiteorder} after the statement for the formal definition

\begin{theorem}\label{thm:gauss_bonn_lim_intro0}
            Let $(M,\mathcal D,g)$ be a contact sub-Riemannian manifold and $S$ be an embedded, compact and orientable surface, with characteristic vector field $X$. Denote $K^\ve$ and $\sigma^\ve$ the Gaussian curvature and area form induced on $S$ by the metric $g^\ve$, respectively.  
            Assume that the characteristic set consists of points which are of finite order of degeneracy. Then the following limit of signed Borel measures exists in the strong $C^{1}(S)^{*}$ sense
            \begin{equation}\label{eq:Omega_-1}
                \mu_{-1}:=\lim_{\ve\to 0}\sqrt{\ve}K^\ve\sigma^\ve. 
            \end{equation}
           Furthermore  the characteristic set is finite $\Sigma(S)=\{q_1,\dots,q_\ell\}\subset S$ and the following limit  exists in the strong $C^{1}(S)^{*}$ sense
            \begin{equation}\label{eq:limit_measure_intro}
                \lim_{\ve\to 0}\left(K^\ve\sigma^\ve-\frac{\mu_{-1}}{\sqrt{\ve}}\right)=2\pi\sum_{i=1}^\ell\mathrm{ind}(q_i)\delta_{q_i},
            \end{equation}
            where $\delta_{q_{i}}$ denotes a Dirac delta centred at $q_i\in S$ and $\mathrm{ind}(q_i)$ denotes the index of $X$ at $q_i$.
        \end{theorem}

\begin{remark} We stress that characteristic points of finite order of degeneracy are automatically isolated (cf.\ Theorem~\ref{thm:finite_order_intro}), therefore, by compactness of $S$, $\Sigma(S)$ is finite. For non-compact surfaces $S$, characteristic points with finite order degeneracy form a discrete set, hence one can obtain the validity of the statement of Theorem~\ref{thm:gauss_bonn_lim_intro0}  by considering the convergence in $C^{1}_{c}(S)^{*}$, i.e., with test functions that have compact support. We stress that in this case \eqref{eq:limit_measure_intro} in particular implies that for $\varphi\in C^{1}_{c}(S)$ one has
  \begin{equation}\label{eq:limit_measure_intro_bis}
                \lim_{\ve\to 0}\int_{S}\varphi \left(K^\ve\sigma^\ve-\frac{\mu_{-1}}{\sqrt{\ve}}\right)=2\pi\sum_{i=1}^\ell\varphi(q_{i})\mathrm{ind}(q_i),
            \end{equation}
            the sum in the right hand side being finite since $\varphi$ has compact support.
\end{remark}
\begin{remark} The measure $\mu_{-1}$ can be computed explicitly as in \cref{rmk:mu_-1}. Notice that the Riemannian Gauss-Bonnet theorem implies that for every $\varepsilon>0$ one has
$$\sqrt{\ve}\int_S K^\ve \sigma^\ve=2\pi \sqrt{\ve}\chi(S)$$
where $\chi(S)$ is the Euler characteristic of $S$. Hence taking the limit when $\ve \to 0$ one gets by \eqref{eq:Omega_-1} the following consequence, thus recovering the statements of \cite{Balogh1} (cf.\ also \cite{V23} and \cite{grong2024})
\begin{equation} \label{eq:int0}
\int_{S}\mu_{-1}=0.
\end{equation}\end{remark}
 
        \subsection{Characteristic points of finite order} \label{ss:finiteorder}
       
       The notion of characteristic point of finite order of degeneracy, which we discuss here, extends the concept of mildly degenerate characteristic point introduced in \cite{Rossi_deg} for the specific case of the Heisenberg group. This extension is not immediate, since in the Heisenberg case the definition relies in an essential way on the existence of a Lie group structure. In the general setting different techniques are required. 
       
       Our approach not only permits the extension of the notion to arbitrary spaces, but also yields finer structural results. Indeed, we show that points of finite order of degeneracy can only occur isolated, and we give sharp estimates on the norm of the characteristic vector field in their vicinity. Finally, we compute the index of the characteristic vector field at such singularity in terms of invariants.
       
        We start by defining horizontal kernel extensions at degenerate characteristic points, i.e., characteristic points $q$ where $\det (D_qX)=0$.
         \begin{defi}
        	Let $S$ be a surface embedded in a 3D contact sub-Riemannian manifold, with char\-acteristic vector field $X$. Let $q\in S$ be a degenerate characteristic point. A \emph{horizontal kernel extension} at $q$ is an arc-length parametrized horizontal smooth curve $\gamma:(-\ve,\ve)\to S$ satisfying 
        	\begin{equation}
        		\gamma(0)=q,\qquad \dot\gamma(0)\in \ker(D_q X).
        	\end{equation}
        \end{defi}
        Recall by Remark~\ref{r:ker1} that the char\-acteristic vector field $X$ at degenerate characteristic points satisfies $\det (D_qX)=0$ and $\dim \ker (D_q X)=1$. Hence by a direct application of the center manifold theorem one can prove the following result  (a more detailed proof is given in \cref{sec:mildly_deg}).
   \begin{lemma} \label{l:novod} 
  	Let $q\in S$ be a degenerate characteristic point, i.e., $\det (D_qX)=0$. Then a horizontal kernel extension $\gamma:(-\ve,\ve)\to S$ always exists.
   \end{lemma}
      In general, a horizontal kernel extension is not unique. The definition of order of degeneracy of a characteristic point is based on the following observation.  
        \begin{prop} \label{prop:critical_curve_lambda_intro}
        	Let $q\in S$ be a degenerate characteristic point, i.e., $\det (D_qX)=0$. Then 
	\begin{itemize}
	\item[(i)] the order of vanishing\footnote{ Given a function $f:\R\to \R$ such that $f(0)=0$, the order of vanishing of $f$ at $0$ is $k:=\inf\left\{i\in\mathbb N : f^{(i)}(0)\neq 0 \right \}$.} at $t=0$ of the following smooth function
        	\begin{equation}\label{eq:lambda_0_intro}     
        		\lambda_\gamma:(-\ve,\ve)\to\mathbb R,\qquad \lambda_\gamma(t)=g(X,\dot\gamma(t)),
        	\end{equation}
        	 does not depend on the choice of the horizontal kernel extension $\gamma$, 
	\item[(ii)] if $\lambda_\gamma$ has finite order of vanishing $k\in \N$ at $t=0$, with $k$ odd, then the following quantity         
		\begin{equation}\label{eq:k-th_eigenvalue_intro}
        		\Lambda^{(k)}(q):=\left.\frac{d^k}{dt^k}\lambda_\gamma(t)\right|_{t=0},
        	\end{equation}
	does not depend on the choice of the horizontal kernel extension.
	\end{itemize}
        \end{prop}
     \begin{defi}
    	We define the \emph{order of degeneracy} of the characteristic point $q$, denoted $\mathrm{ord}(q)$, as the order of vanishing at $t=0$ of the function $\lambda_\gamma$  in \eqref{eq:lambda_0_intro} for some horizontal kernel extension $\gamma$. 
	\end{defi}
	\begin{remark}
		We highlight here that if the order of degeneracy of $q$ is an odd integer $k\in\mathbb N$, the quantity $\Lambda^{(k)}(q)$ defined in \eqref{eq:k-th_eigenvalue_intro} is well-defined, while if the order of degeneracy of $q$ is an even integer, the quantity $\Lambda^{(k)}(q)$ is well-defined only up to a sign, depending on the orientation of the initial velocity of the kernel extension alone (cf.~Section \ref{sec:mildly_deg}). For even $k$, we fix such orientation and the positive representative of $\Lambda^{(k)}(q)$, as explained in Remark~\ref{rmk:Lambda_sign}, so that \eqref{eq:k-th_eigenvalue_intro} continues to hold.
	\end{remark}	
	We stress that for a degenerate characteristic point $q$ one necessarily  has $\mathrm{ord}(q)\geq 2$, cf. Lemma~\ref{l:almeno2}. 	
	By abuse of notation, we set $\mathrm{ord}(q)= 1$ for non-degenerate characteristic points $q$.  We say that a characteristic point has finite order of degeneracy if $\mathrm{ord}(q)<+\infty$.
	
   Characteristic points of finite order of degeneracy satisfy the following properties.
        \begin{theorem}\label{thm:finite_order_intro}
                Let $S$ be an oriented surface embedded in a 3D contact sub-Riemannian manifold, with characteristic vector field $X$. Let $q\in S$ be a characteristic point of order of degeneracy $k\geq 2$, then:
                \begin{itemize}
                    \item[(i)] $q$ is an isolated characteristic point,
                    \item[(ii)] there exist local coordinates $(x,y)$ near $q$ and a positive constant $C>0$ such that 
                \begin{equation}
                    q=(0,0),\qquad C^{-1}\sqrt{x^2+y^{2k}}\leq|X|\leq C\sqrt{x^2+y^{2k}},
                \end{equation} 
                \item[(iii)] the index of $q$ satisfies 
                \begin{equation}\label{eq:index_formula}
                    \mathrm{ind}(q)=\begin{cases}
                        \mathrm{sign}\left(\mathrm{div}_q(X)\Lambda^{(k)}(q)\right), & \text{if}\,\,k\,\,\text{odd},\\
                        0,\qquad &\text{if}\,\,k\,\,\text{even}.
                    \end{cases}
                \end{equation}
                \end{itemize}
            \end{theorem}
        
         We stress that characteristic points with finite order of degeneracy are isolated. In the analytic case, the converse also holds: isolated characteristic points have finite order of degeneracy, cf.\,\cref{prop:curve_of_char}.
\begin{remark}
            We stress that in Theorem~\ref{thm:finite_order_intro} one has $k\geq 2$. Actually one can include the case of non-degenerate characteristic points, i.e., when $k=1$. Indeed claims (i) and (ii) hold thanks to $\det(D_{q}X)\neq 0$, while in (iii) one should  suitably reinterpret the quantity $\Lambda^{(k)}(q)$ for $k=1$ as follows
            \begin{equation} \label{eq:rapporto}
            \Lambda^{(1)}(q)=\frac{\det(D_q X)}{\mathrm{div}_q (X)}=\frac{\det(D_q X)}{\mathrm{trace}( D_q X)}.
            \end{equation}
\end{remark}
 All the quantities $\Lambda^{(k)}(q)$ for all $k\geq 1$ can be defined and computed in a unified way through an auxiliary affine connection as follows.

 \begin{prop}\label{prop:v_0_intro}
 	Let $q\in S$ be a characteristic point of order  $k\geq 2$. Let $\nabla$ be the Levi-Civita connection on $TS$ associated with $g^1|_S$. There exists a neighbourhood $U\subset S$ of $q$ such that the following linear endomorphism is (fibrewise) diagonalizable
 	\begin{equation}
 		\nabla X:TS|_U\to TS|_U,\qquad Y\mapsto \nabla_YX.
 	\end{equation}
  In particular $\nabla X$ admits an eigenvector field $v_0\in\mathfrak X(U)$ satisfying $v_0|_q\in\ker D_q X$ and $|v_0|_q=1$. For even $k$, we choose $v_0|_q$ to be positively oriented (cf. \cref{rmk:Lambda_sign}). Then
 	\begin{equation}\label{eq:lambdagrande}
 		\Lambda^{(k)}(q)=\left.v_0^{(k-1)}\left(\frac{\det(\nabla X)}{\mathrm{trace}(\nabla X)}\right)\right|_q.
 	\end{equation}
 	Furthermore, $k$ is the smallest natural number such that the right-hand side of \eqref{eq:lambdagrande} is non zero.
 \end{prop}
 \noindent
 Notice that for $k=1$ the vector $v_{0}$ is not defined but \eqref{eq:lambdagrande} makes sense and coincides with \eqref{eq:rapporto}.
 
As it will be clear from the proof (see Section~\ref{app:formula_lambda}), the previous result is actually independent of the choice of the affine connection $\nabla$ used to compute \eqref{eq:lambdagrande}.
 
 \subsection{Consequences: integrability} 
  A direct consequence of item $(ii)$ of \cref{thm:finite_order_intro}, being the function $(x^2+y^{2k})^{-1/2}$ integrable near the origin in $\R^{2}$ for $k\geq 1$, is the following.
\begin{corollary}\label{c:garofalo}
	Let $S$ be a surface embedded in a 3D contact sub-Riemannian manifold with a smooth measure $\sigma$. Let $X$ be the characteristic vector field. Assume all characteristic points have finite order of degeneracy, then
	\begin{equation}\label{eq:L^1_cond_body}
		\frac{1}{|X|}\in L^1_{loc}(S,\sigma).
	\end{equation}
\end{corollary}
This theorem extends \cite[Theorem~4.5]{Rossi_deg}, whose proof exploits the isometries of the Heisenberg group. Our approach relies on the normal form of \cref{lem:normal_form_strong_body}.

\cref{c:garofalo} is related with integrability properties of the sub-Riemannian mean curvature $\mathcal{H}$ of $S$, which can be defined as the limit of mean curvatures of the Riemannian approximations, see \cite{BB24}. The estimate $|\mathcal{H}|\leq C |X|^{-1}$ follows from a technical but standard observation, cf.\ \cite[Section 2.1]{Rossi_deg} and \cite[Prop. 3.5]{DNG}. In the analytic case, we prove the following.
\begin{corollary} \label{c:garofalo2}
	Let $S$ be an analytic surface embedded in an analytic 3D contact sub-Riemannian manifold with a smooth measure $\sigma$. 
	Let $\mathcal{H}$  be the sub-Riemannian mean curvature. If every characteristic point is isolated then
	\begin{equation}\label{eq:L^1_cond_body}
		\mathcal{H}\in L^1_{loc}(S,\sigma).
	\end{equation}
\end{corollary}
The previous result extends \cite[Theorem 4.7]{Rossi_deg} concerning a conjecture in \cite{DNG}.

\subsection{Consequences: Euler characteristic}
Combining Poincar\'e-Hopf theorem with $(iii)$ of \cref{thm:finite_order_intro} we can compute the Euler characteristic of a surface in terms of the  invariants $\Lambda^{(k)}(q)$.
\begin{corollary}\label{c:chidiv}
    Let $S$ be a compact oriented surface embedded in a 3D contact sub-Riemannian manifold. Assume that every characteristic point has finite order of degeneracy. Then the characteristic set is finite $\Sigma(S)=\{q_1,\dots,q_\ell\}$, and 
    \begin{equation}\label{eq:lanuova}
        \chi(S)=\sum_{k_i\,\mathrm{odd}}\mathrm{sign}\left(\mathrm{div}_{q_i}(X)\Lambda^{(k_i)}(q_i)\right),
    \end{equation}
    where $k_i=\mathrm{ord}(q_i)$, and $\Lambda^{(k_i)}(q_i)$ are the invariants defined in \eqref{eq:lambda_0_intro}-\eqref{eq:k-th_eigenvalue_intro}.
\end{corollary}
 
\begin{remark} We stress that if all characteristic points are non-degenerate, i.e., $k_i=1$ for all $i=1,\ldots,\ell$, then, thanks to \eqref{eq:rapporto}, formula \eqref{eq:lanuova} reduces to the classical identity
\begin{equation}\label{eq:classico}
\chi(S)=\sum_{i=1}^\ell\mathrm{sign}\left(\det (D_{q_i}X)\right).
    \end{equation}
\end{remark}

 Going back to our main Theorem~\ref{thm:gauss_bonn_lim_intro0}, the atomic measure appearing in \eqref{eq:limit_measure_intro} is obtained as limit of the Gaussian curvature measures $K^\ve\sigma^\ve$ of the Riemannian surfaces $(S,g^\ve|_S)$. It is therefore natural to compare their masses with the pointwise limit of the Gaussian curvature $K^\ve$ at characteristic points. 
 
 However, it is known that for $\ve\to 0$ the Gaussian curvature $K^\ve$ at characteristic points diverges, and it converges only after an appropriate rescaling. More precisely, let $q\in S$ be a characteristic point, it was shown in \cite{Daniele} that 
                \begin{equation}
                    \widehat K(q):=\lim_{\ve\to 0}{\ve} K^\ve(q)=-1+\frac{\det(D_qX)}{\mathrm{trace}(D_qX)^2}.
                \end{equation}
One can compare such a limit with our results by extending the function $ \widehat K$ near characteristic points through the following definition
\begin{equation}\label{eq:hat_K}
    \widehat K(q)=-1+\frac{\det (\nabla X|_{q})}{\mathrm{trace}^2(\nabla X|_{q})},
\end{equation}
where $\nabla$ is the Levi-Civita connection of the metric $g^1|_S$. In particular, restating Corollary~\ref{c:chidiv}, one can compute the Euler characteristic in terms of this curvature invariant $\widehat K$ and its derivatives.
\begin{prop}
   Let $S$ be an oriented surface embedded in a 3D contact sub-Riemannian manifold with characteristic vector field $X$. Let $q\in S$ be a characteristic point of finite odd order of degeneracy $k\in \N$, and let $v_0$ be the eigenvector field defined in \cref{prop:v_0_intro}. Then 
    \begin{equation}
        \mathrm{div}_q(X)\Lambda^{(k)}(q)=v_0^{(k-1)}\left(\widehat{K}+1\right)\Big|_{q},
    \end{equation}
    where $\widehat K$ is the function defined in \eqref{eq:hat_K}. If $S$ is compact and all characteristic points are of finite order of degeneracy, then 
    \begin{equation}\label{eq:chiodd}
                    \chi(S)=\sum_{k_i\,\,\text{odd}}\mathrm{sign}\left(v_0^{(k_i-1)}\left(\widehat{K}+1\right)\Big|_{q_{i}}\right),
                \end{equation} 
                where $\{k_1,\dots,k_\ell\}$ are the orders of the characteristic points $\{q_1,\dots,q_\ell\}$.
\end{prop}  
Again, we stress that if all characteristic points are non-degenerate  $k_{i}=1$ for all $i$ and the formula \eqref{eq:chiodd} reduces to \eqref{eq:classico} since $\mathrm{sign}(\widehat{K}+1)=\mathrm{sign}(\det D_{q}X)$.

 \subsection{An example: the horizontal plane in the Heisenberg group} It is interesting to illustrate the discrete contribution to the limiting measure in Theorem~\ref{thm:gauss_bonn_lim_intro0} for the explicit example of the plane $\{z=0\}$ in the three-dimensional Heisenberg group.
 
         Let $(\mathbb R^3,\mathcal D,g)$ be the Heisenberg group: the contact distribution $\mathcal D$ is expressed as the kernel of a differential form $\omega$ which, in cylindrical coordinates $x=r\cos\theta, y=r\sin\theta, z$, reads 
     	\begin{equation}\label{eq:example_form_metrics0}
     		\omega=dz+\frac{r^2}{2}d\theta.
     	\end{equation}
     	The extended metrics \eqref{eq:g_ve_properties0}, whose restrictions to $\mathcal D$ coincide with the sub-Riemannian metric $g$, have the expression 
     	\begin{equation}
     		g^\ve=dr^2+r^2d\theta^2+\frac{1}{\ve}\omega\otimes \omega.
     	\end{equation}
     	Let $S=\{(x,y,z)\mid z=0\}\subset \R^{3}$. Substituting the expression of $\omega$ into the one of $g^\ve$ and imposing $dz=0$ we obtain
     	\begin{equation}
     		g^\ve|_S=dr^2+f_\ve(r)^2d\theta^2,\qquad\text{where}\qquad  f_\ve(r)=r\sqrt{1+\frac{r^2}{4\ve}}.
     	\end{equation}
     	A standard computation shows that 
     	\begin{equation}\label{eq:acca1}
     		K^\ve\sigma^\ve=-f''_\ve(r)drd\theta=-\frac{r \left(r^2 + 6\varepsilon\right)}{\sqrt{\varepsilon} \left(r^2 + 4\varepsilon\right)^{3/2}}drd\theta.
     	\end{equation}
     	One can then explicitly  compute the limit 
     	\begin{equation}\label{eq:acca2}
     		\mu_{-1}:=\lim_{\ve\to0}\sqrt{\ve}K^\ve\sigma^\ve=-drd\theta.
     	\end{equation}
	In this specific example, the statement of Theorem~\ref{thm:gauss_bonn_lim_intro0} can be proved by a direct computation and shows that the discrete part of the measure is a delta in the origin, which is characteristic.
\begin{lemma}
	For every $\varphi\in C^1_{c}(\mathbb R^2)$, the following equality holds 
	\begin{equation}
		\lim_{\ve\to 0}\int_{\mathbb R^2}\varphi\left(K^\ve\sigma^\ve-\frac{\mu_{-1}}{\sqrt\ve}\right)
		=2\pi\varphi(0),
	\end{equation}
	where the measures $K^\ve\sigma^\ve$ and  $\mu_{-1}$ are defined as in \eqref{eq:acca1} and \eqref{eq:acca2}
\end{lemma}
The proof is an explicit computation and is contained in Appendix~\ref{a:heis}.  
\subsection{Comparison with the literature} We compare \cref{thm:gauss_bonn_lim_intro0} with \cite[Theorem~1.1]{Balogh1} and 	\cite[Theorem~1.1]{grong2024}, in which a sub-Riemannian Gauss-Bonnet formula is investigated through similar methods. We remark that, although \cite{Balogh1,grong2024} work with $C^2$ surfaces, the comparison described in this section is intended within the context of this paper, i.e. in the smooth category.

	The Lebesgue absolutely continuous Radon measure $\mu_{-1}$ found in \eqref{eq:Omega_-1} was already identified, in the Heisenberg group, in \cite{Balogh1}. In particular, identity \eqref{eq:int0} corresponds to the closed surface case of the Gauss-Bonnet formula of \cite[Theorem~1.1]{Balogh1}.
	Their hypotheses include
	$\mathcal H_E^1(\Sigma(S))=0$ and local integrability of
	$|\nabla_Hu|^{-1}$ with respect to Euclidean surface area,
	where $u$ is a local defining function.
	The latter condition is locally equivalent to
	integrability of $|X|^{-1}$ with respect to $\sigma^1$.
	Our finite order hypothesis implies that the characteristic set
	is finite and, by \cref{c:garofalo}, that $|X|^{-1}$ is locally
	integrable, thus it provides a geometric sufficient condition for the hypotheses of \cite[Theorem~1.1]{Balogh1}.
	
	For surfaces in general three-dimensional contact sub-Riemannian manifolds, in \cite[Theorem~1.1]{grong2024}
	the authors recover the Euler characteristic from the derivative of a curvature
	integral over shrinking neighbourhoods of the characteristic set. Their Assumption~(A) implies that every characteristic
	point is non-degenerate, as shown in \cref{rmk:assumption_A}.
	It therefore requires $\mathrm{ord}(q)=1$ for every
	$q\in\Sigma(S)$, while our finite-order
	hypothesis allows $\mathrm{ord}(q)<+\infty$.

        \subsection{Structure of the paper}
        In Section~\ref{s:sicsm} we review some preliminaries on contact sub-Riemannian manifolds and basic facts concerning surfaces embedded in contact manifolds. In \cref{sec:mildly_deg} we introduce characteristic points of finite order of degeneracy and in \cref{ssec:order_main_thm} we derive their properties. In \cref{sec:ve-frames} we introduce the main objects and lemmas needed in the proof of \cref{thm:gauss_bonn_lim_intro0}, which is presented in \cref{sec:proof_of_main}. 
     
     \smallskip

	{\bf Acknowledgements.} Davide Barilari acknowledges the support granted by the MUR-PRIN 2022 project ``Optimal Transport: new challenges across analysis and geometry'' (Project code: 2022K53E57). Andrea Pinamonti was partially supported by the GNAMPA project \emph{Structures of sub-Riemannian hypersurfaces in Heisenberg groups}. Further support was provided by MUR and the University of Trento (Italy) through the MUR-PRIN 2022 project \emph{Regularity problems in Sub-Riemannian structures} (Project code: 2022F4F2LH).  
Andrea Pinamonti is member of the \emph{Gruppo Nazionale per l'Analisi Matematica, la Probabilit\`a e le loro Applicazioni} (GNAMPA), which is part of the \emph{Istituto Nazionale di Alta Matematica} (INdAM). We thank the anonymous referees for their careful reading and detailed comments.
	
        \section{Surfaces in 3D contact sub-Riemannian manifolds}\label{s:sicsm}
         We recall some preliminaries on contact sub-Riemannian manifolds and fix our notation. We refer to \cite{Geiges} and \cite{Agrachev} for introductions to contact and sub-Riemannian geometry, respectively.
     \subsection{Contact sub-Riemannian manifolds}\label{sec:preliminaries}   
      
A three-dimensional contact sub-Riemannian manifold is a triple $(M,\mathcal{D},g)$, where $M$ is a smooth three-dimensional manifold, $\mathcal{D}$ is a \emph{contact distribution} and $g$ a \emph{sub-Riemannian metric}, i.e., a smooth metric on $\mathcal D$. We recall that a plane field $\mathcal{D}$ is a contact distribution if locally $\mathcal D=\ker \omega$ for some one-form $\omega$ satisfying $d\omega|_{\mathcal D}$ is non-degenerate. In dimension three this is equivalent to requiring $\mathcal D +[\mathcal D,\mathcal D]=TM$.

          A sub-Riemannian contact manifold is called co-orientable if the one-form satisfying $\mathcal D=\ker \omega$ can be chosen globally. In this case the three-form $\omega\wedge d\omega$ is a global volume form and $M$ is also orientable.  For simplicity in what follows we always assume the contact structure to be co-oriented.

    The Reeb vector field associated with a contact form $\omega$ is the vector field $f_0\in\mathfrak X(M)$ uniquely determined by the following two equations 
    \begin{equation}\label{eq:def_reeb}
        \omega(f_0)=1,\qquad d\omega(f_0,\cdot)=0.
    \end{equation}
  The contact form $\omega$ is called normalized with respect to the sub-Riemannian metric if it satisfies
    \begin{equation}\label{eq:omega_norm}
        d\omega(f_{1},f_{2})=1,
    \end{equation}
    for every positively oriented orthonormal frame for $\mathcal D$ (i.e., $d\omega|_{\mathcal D}$ is the area form defined by the metric $g$). 
    Without further notice, in the rest of the paper we always assume $\omega$ to be normalized.
    
    Given a contact sub-Riemannian manifold one can define the length of an absolutely continuous curve $\gamma:[0,1]\to M$ which is \emph{horizontal} (i.e., almost everywhere tangent to $\mathcal D$) as
    \begin{equation}
        \mathrm{length}(\gamma)=\int_{0}^1\|\dot\gamma(t)\|_gdt,
    \end{equation}
    and the sub-Riemannian distance $d_{SR}:M\times M\to\mathbb R$ as follows
    \begin{equation}
        d_{SR}(q_0,q_1)=\inf\{\mathrm{length}(\gamma)\mid \omega(\dot\gamma)=0,\,\gamma(0)=q_0,\,\gamma(1)=q_1\},\qquad \forall\,q_0,q_1\in M.
    \end{equation}
    where the infimum is over all horizontal absolutely continuous curve joining the two points.
    It is well-known that $d_{SR}$ is actually a distance function since $\mathcal D$ is bracket-generating, see \cite[Ch.\ 3]{Agrachev}.

    Any sub-Riemannian space can be seen as a Gromov-Hausdorff limit of Riemannian spaces. In the contact sub-Riemannian case there exists a special family of approximating Riemannian metrics.
    \begin{defi}\label{def:g_ve}
            Let $(M,\mathcal D, g)$ be a contact sub-Riemannnian manifold and let $f_0$ be the associated Reeb vector field. For any $\ve>0$ we define the Riemannian metric $g^\ve$ as the unique one satisfying 
            \begin{equation}\label{eq:g_ve_properties}
                g^\ve|_{\mathcal D}=g,\qquad g^\ve(\mathcal D,f_0)=0,\qquad g^\ve(f_0,f_0)=\frac{1}{\ve}.
            \end{equation}  \end{defi}
            Furthermore, if $S\subset M$ is an embedded oriented surface, we denote by $\sigma^\ve$ the area form of the oriented Riemannian surface $(S,g^\ve|_S)$, i.e., the differential two-form satisfying 
                        \begin{equation}\label{eq:def_sigma_ve}
               \sigma^\ve(v^{\ve}_{1},v^{\ve}_{2})=1.
            \end{equation}
      for every positively oriented orthonormal frame $v^{\ve}_{1},v^{\ve}_{2}$ of $S$.

       \subsection{Surfaces in three-dimensional contact sub-Riemannian manifolds}\label{sec:surfaces_in_3D_cont}
 Let $S$ be a surface embedded in a co-oriented three-dimensional contact manifold $(M,\mathcal D)$. The intersection $TS\cap\mathcal D$  defines a field of directions, whose singularities are the points $p\in S$ where $T_pS=\mathcal D_p$, and are called characteristic points. The characteristic set $\Sigma(S)$ is the set of characteristic points. We recall that $\Sigma(S)$ is contained in a one-dimensional submanifold of $S$, hence has measure zero \cite{Daniele}.
 
 Let $\omega\in\Omega^1(M)$ be a contact form. If the surface is oriented by some area form $\sigma\in \Omega^2(S)$, then we can define the characteristic vector field as the unique vector field $X\in \mathfrak X(S)$ satisfying
        \begin{equation}\label{eq:general_X}
            \iota_X\sigma=\omega|_S.
        \end{equation}
        The characteristic vector field vanishes at characteristic points:
        \begin{equation}
            \mathrm{span}_\mathbb R\{X_p\}=\begin{cases}
                T_pS\cap \mathcal D_p,\quad &p\notin \Sigma(S),\\
                0,&p\in \Sigma(S).
            \end{cases}
        \end{equation}
        For a given surface, there exist several different characteristic vector fields. In fact, they constitute a conformal class: rescaling the area form $\sigma$ by a positive function results in a rescaling of $X$ by the reciprocal of the same function. 
        \begin{remark}\label{rmk:div_non_zero}
            Any characteristic vector field has non zero divergence at characteristic points. Recall that, the divergence of $X$ with respect to the area form $\sigma$ is defined by ($L_X$ denoting the Lie derivative)
        \begin{equation}\label{eq:div(X)_formula0}
            \mathrm{div}(X)\sigma=L_X\sigma=d\iota_X \sigma.
        \end{equation}
        Therefore, taking the differential of both sides of \eqref{eq:general_X} and using $L_X=d\circ \iota_X +\iota_X\circ d$ we obtain
            \begin{equation}\label{eq:div(X)_formula}
                \mathrm{div}(X)\sigma=d\omega|_S.
            \end{equation}
            Now, if $q\in S$ is a characteristic point, then
       $T_qS=\mathcal D_q$.
            By the contact condition, the restriction of $d\omega$ to $\mathcal D$ is non degenerate, therefore \eqref{eq:div(X)_formula} implies that $\mathrm{div}(X)$ does not vanish at $q$.
        \end{remark}
        In the case of a surface embedded in a contact sub-Riemannian manifold we can fix a canonical representative characteristic vector fields by fixing the area form.
        \begin{defi}
             Let $(M,\mathcal D,g)$ be a contact sub-Riemannian manifold and let $S\subset M$ be an embedded oriented surface. Let $\sigma^1$ be the area form defined in \eqref{eq:def_sigma_ve} for $\ve=1$. We define the sub-Riemannian characteristic vector field $X\in\mathfrak X(S)$ as 
            \begin{equation}\label{eq:def_X}
                \iota_X\sigma^1=\omega|_S,
            \end{equation}
        where $\omega$ denotes the contact form normalized as in \eqref{eq:omega_norm}.
        \end{defi}
    
        The divergence of the sub-Riemannian characteristic vector field and its norm are related.
        \begin{prop}\label{prop:div=pm}
            Let $(M,\mathcal D,g)$ be a contact sub-Riemannian manifold and $S$ an embedded oriented surface.
            Let $\sigma^1$ be the area form defined in \eqref{eq:def_sigma_ve} for $\ve=1$. 
            Then the divergence of the sub-Riemannian characteristic vector field $X$, computed with respect to $\sigma^1$, satisfies on $S$
            \begin{equation}\label{eq:div(X)^2}
                \mathrm{div}(X)^2+|X|^2=1.
            \end{equation}
            \end{prop}
             We stress that in \eqref{eq:div(X)^2} the norm $|\cdot|$ is the sub-Riemannian norm on $\mathcal D$.
            \begin{proof}
            Let us first work on $S\setminus \Sigma(S)$ and set $f_1=X/|X|$ the normalized characteristic vector field, let $f_2$ be the oriented horizontal normal and let $f_0$ be the Reeb field associated with the normalized contact form $\omega$. Notice that, being $g^1|_{\mathcal D}$ equal to $g$, then $f_1$ has norm $1$ for $(S,g^1|_S)$. 
            Let $J:TS\to TS$ be the orthogonal map defined by 
            \begin{equation}\label{eq:JS}
            \sigma^1(v,Jw)=g^1(v,w),\qquad v,w\in TS.
            \end{equation} 
            We define an orthonormal frame for $g^1|_{S\setminus \Sigma(S)}$ as 
            \begin{equation}
            	e_1=f_1,\qquad e_2=Je_1.
            \end{equation}
            Then there exists a function $a:S\to [-1,1]$ such that, 
            \begin{equation}\label{eq:e1,e2}
                e_1=f_1,\qquad e_2=af_2+\sqrt{1-a^2}f_0.
            \end{equation}
            Taking the differential of the 1-forms in both sides of \eqref{eq:def_X} and evaluating the result at $e_1,e_2$ we get 
            \begin{equation}\label{eq:divea}
                \mathrm{div}(X)=d\omega(e_1,e_2)=d\omega(f_1,af_2+\sqrt{1-a^2}f_0)=ad\omega(f_1,f_2)=a,
            \end{equation}
            where, in the third equality we have used that $d\omega(f_{0},\cdot)=0$ from  \eqref{eq:def_reeb}, and in the last one the normalization condition \eqref{eq:omega_norm}.
            Moreover, using again \eqref{eq:JS}, we can compute
            \begin{equation}
                \begin{aligned}
                |X|^2&=g^1(X,X)=\sigma^1(X,JX)=\iota_X\sigma^1(JX)=\omega(JX)\\
                &=|X|\omega(af_2+\sqrt{1-a^2}f_0)=|X|\sqrt{1-a^2},
                \end{aligned}
            \end{equation}
            where the fourth equality follows from \eqref{eq:def_X}, and the fifth from \eqref{eq:e1,e2} recalling that $e_1=f_1=X/|X|$. This implies $|X|=\sqrt{1-a^2}$ and, being
        $a=\mathrm{div}(X)$ thanks to \eqref{eq:divea}, identity \eqref{eq:div(X)^2} is proved on $S\setminus \Sigma(S)$. By continuity, the equality is proved on $S$.
            \end{proof}
            We recall the definition of differential of a vector field at a singular point.
            \begin{defi}
                Let $S$ be a smooth manifold and $X\in\mathfrak X(S)$ a smooth vector field vanishing at $q\in S$. The differential of $X$ at $q$ is the map 
                \begin{equation}
                    D_qX:T_qS\to T_qS,\qquad D_qX(Y)=[Y,X]_q.
                \end{equation}
            \end{defi}
       
                Notice that the definition of $D_qX(Y)$ requires a choice of a vector field extension of $Y$, near $q$. However, since $X_q=0$, the vector $[Y,X]_q$ only depends on $Y_q$.
 			 \begin{remark}\label{r:tresei}
 				Let $\nabla$ be the Levi-Civita connection on $TS$ for $g^{1}|_{S}$, and $q$ a characteristic point, then 
 				$$\nabla_{Y}X|_{q}=\nabla_{X}Y|_{q}+[Y,X]|_{q}=(D_{q}X)Y,\qquad \forall \,Y\in {T_qS},$$ 
 				where we used that $\nabla_{X}Y|_{q}=0$ since $X|_{q}=0$.  At a characteristic point $q$ we have the equality 
               \begin{equation}
               	\mathrm{div}_q(X)=\mathrm{trace}(D_qX)=\mathrm{trace}(\nabla X|_{q}).
               \end{equation}
               By \cref{rmk:div_non_zero} we deduce that $\mathrm{trace}(D_qX)\neq 0$ and that $\rank D_qX\geq 1$ at characteristic points.

            \end{remark}
            {\begin{remark}\label{rmk:assumption_A}
            		The function $a=\mathrm{div}(X)$ appearing in the proof of \cref{prop:div=pm}, is the same appearing in Assumption~(A) of \cite{grong2024}. This assumption requires the functions
            		\begin{equation}
            			\Phi_{\pm}:(0,\delta)\to \mathbb R,\qquad \Phi_{\pm}(c)=\int_{\{1\pm a\leq c\}}\sigma^1,
            		\end{equation}
            		to be differentiable on the interval $(0,\delta)$ with uniformly bounded derivative, for sufficiently small $\delta>0$. We observe that (A) implies that every characteristic point is non-degenerate. Indeed, (A) implies that the following function 
            		\[
            		\Phi:(0,\delta)\to \R,\qquad
            		\Phi(c):=\Phi_+(c)+\Phi_-(c)=\int_{\{1-|a|\leq c\}}\sigma^1,
            		\]
            		is Lipschitz, being the sum of two functions with uniformly bounded derivatives. Thus, $\Phi$ extends to a Lipschitz function on $[0,\delta]$. Moreover, by \cref{prop:div=pm}, $|a(q)|=1$ if and only if $q\in\Sigma(S)$. Since the characteristic set has zero area, we have $\Phi(0)=0$. Then, since $\Phi$ is Lipschitz on $[0,\delta]$ we have
            		\begin{equation}
            			\Phi(c)=O(c),\qquad \text{as $c\to 0^+$}.
            		\end{equation}
            		The function $f:=1-|a|$ is smooth near $f^{-1}(0)$ (by $a=\mathrm{div}(X)$ and \cref{r:tresei}). Thus, an application of \cref{lem:sublevel_nondegeneracy} implies that any zero of $f$ is a non-degenerate minimum. By \cref{prop:div=pm} the zeroes of $f$ coincide with the characteristic points and, a standard computation using \eqref{eq:div(X)^2} implies that 
            		\[
            		\rank(\operatorname{Hess}_q f)
            		=\rank(D_qX),\qquad \forall q\in \Sigma(S).
            		\]
            		It follows that, if (A) holds, then all characteristic points are non-degenerate. The converse implication is proved in \cite{grong2024}.
            \end{remark}}
            \section{Order of degeneracy of characteristic points. Proof of Proposition~\ref{prop:critical_curve_lambda_intro}}\label{sec:mildly_deg}

            Throughout this subsection $S$ denote an oriented surface embedded in a contact sub-Riemannian manifold $(M,\mathcal D,g)$ with sub-Riemannian characteristic vector field $X$.          
            
We start by proving the existence of a horizontal kernel extension $\gamma:(-\varepsilon,\varepsilon)\to S$.
    	\begin{proof}[Proof of Lemma~\ref{l:novod}]
    
    According to \cref{rmk:div_non_zero}, it holds $\mathrm{div}_q(X)=\mathrm{trace}(D_qX)\neq 0$. Therefore, the characteristic polynomial of the endomorphism $D_qX:T_qS\to T_qS$ reads
		\begin{equation}
			P(\lambda)=\lambda(\lambda-\mathrm{div}_q(X)).
		\end{equation} 
		Applying the classical center manifold theorem (see for example \cite[Thm.\,1, Sec.\,2.12]{Perko2001}) we deduce the existence of a smooth embedded submanifold $\mathcal C$ which is invariant under the flow of $X$ and such that $T_q\mathcal C=\ker D_q X$. Let $\gamma:(-\ve,\ve)\to S$ be a smooth regular parametrization of $\mathcal C$. Since $\mathcal C$ is invariant under the flow of $X$, we deduce that $\gamma$ is horizontal. Reparametrizing $\gamma$ by arc-length we find a horizontal kernel extension. 
	\end{proof}
   \begin{lemma}\label{l:almeno2}
   Given a horizontal kernel extension $\gamma:(-\ve,\ve)\to S$, the following smooth function
    \begin{equation}\label{eq:lambda_N_body}        
    	\lambda_\gamma:(-\ve,\ve)\to\mathbb R,\qquad \lambda_\gamma(t)=g(X,\dot\gamma(t))|_{\gamma(t)},
    \end{equation}
satisfies $\lambda_{\gamma}(0)=\dot \lambda_{\gamma}(0)=0$, i.e., has order of vanishing at zero $\geq 2$.
   \end{lemma}
   \begin{proof} Since at characteristic points $X$ vanishes, then $\lambda_{\gamma}(0)=0$. Denoting $\nabla$ the Levi-Civita connection of $g^1|_S$, we compute
   $$\frac{d}{dt}g(X,\dot\gamma)|_{\gamma(t)}=g(\nabla_{\dot \gamma(t)}X,\dot\gamma(t))+g(X,\nabla_{\dot \gamma(t)}\dot\gamma(t))$$
  The second term at $t=0$ is zero since $X$ vanishes at characteristic points. The first one vanishes as well:  consider any vector field $Y$ extension of $\dot \gamma(t)$ and compute (at $t=0$)
  \begin{equation}
  	\nabla_{\dot \gamma(t)}X|_{q}=\nabla_{Y}X|_{q}=(D_qX)Y=0,
  \end{equation}
	where we have used \cref{r:tresei} and the fact that $Y|_{q}\in \ker D_{q}X$ by construction.
   \end{proof}
            \subsection{Well-posedness of the order. Proof of Proposition~\ref{prop:critical_curve_lambda_intro}}\label{ssec:defi1}
    We now move to the proof of the fact that the order of degeneracy is well-defined.
	\begin{proof}[Proof of Proposition~\ref{prop:critical_curve_lambda_intro}]
	
	Let $q\in S$ be a degenerate characteristic point, i.e., such that $\det (D_qX)=0$. Recall that a horizontal kernel extension at $q$ is an arc-length parametrized horizontal smooth curve $\gamma:(-\ve,\ve)\to S$ satisfying 
            	$
            		\gamma(0)=q$ and $\dot\gamma(0)\in \ker (D_q X)
            	$.

	{(i)}.	We now prove that the order of vanishing of the following smooth function
    \begin{equation}\label{eq:lambda_N_body}        
    	\lambda_\gamma:(-\ve,\ve)\to\mathbb R,\qquad \lambda_\gamma(t)=g(X,\dot\gamma(t))|_{\gamma(t)},
    \end{equation}
    at $t=0$ does not depend on the choice of the horizontal kernel extension $\gamma$. Notice that, once proved this statement, then Lemma~\ref{l:almeno2} implies that $\mathrm{ord}(q)\geq 2$.
    
    Let $g^1$ be the Riemannian metric described in \cref{def:g_ve}. For the sake of simplicity in this proof we denote $g^1|_{S}$ by $g$ and the corresponding area form $\sigma^1$ by $\sigma$. Let $\gamma:(-\ve,\ve)\to S$ be a kernel extension, and let $T\in\mathfrak X(S)$ be a vector field which extends  $\dot\gamma$. Let $N$ denote the oriented orthonormal complement in $S$ with respect to $g$. Then, from \eqref{eq:def_X} we deduce that the restriction of the sub-Riemannian contact form $\omega$ to the surface $S$ can be computed as
		\begin{equation}\label{eq:omega_S}
			\omega|_S(\cdot)=\sigma(X,\cdot)=\det\begin{pmatrix}
				g(T,X) &g(T,\cdot)\\
				g(N,X) & g(N,\cdot)
			\end{pmatrix}.
		\end{equation}
		Let $\mathcal C=\gamma(-\ve,\ve)$. We claim that 
		\begin{equation}\label{eq:C_omega_S}
			\mathcal C=\{p\in S\mid \omega(T)|_p=0\}=\{p\in S\mid g(N,X)|_p=0\}.
		\end{equation}
		We observe that the function $g(N,X)$ is a submersion near $q$. First observe that being $X|_q=0$ and $[T,X]|_q=D_qX(\dot\gamma(0))=0$ one has $\nabla_{T}X|_{q}=0$. For the divergence at the point $q$, one has
		 $$\mathrm{div}_q(X)=\mathrm{trace}(\nabla X|_{q})=g(\nabla_T X,T)|_q+g(\nabla_N X,N)|_q=g(\nabla_{N}X,N)|_q,$$ 
		 where we used \cref{r:tresei}. Thus, differentiating $g(N,X)$ in the direction of $N$ and evaluating the result at $q$ we find
		\begin{equation}
			\begin{aligned}\label{eq:no}
				N(g(N,X))|_q=g(\nabla_{N}N,X)|_q+g(N,\nabla_{N}X)|_q=\mathrm{div}_q(X),
							\end{aligned}
		\end{equation}
where $g(\nabla_{N}N,X)|_{q}=0$ since $X|_{q}=0$. According to \cref{rmk:div_non_zero}, $\mathrm{div}_q(X)\neq 0$. Therefore, the regular value theorem implies that the vanishing set of $\omega(T)$ is a smooth curve near $q$. On the other hand, since $T|_\mathcal C=\dot\gamma$ and $\dot\gamma$ is horizontal, the vanishing set of $\omega(T)$ contains $\mathcal C$; proving \eqref{eq:C_omega_S}.

		Let now ${\gamma}':(-\ve,\ve)\to S$ be another horizontal kernel extension and let ${T}'\in\mathfrak X(S)$ be a vector field extension of its velocity  field $\dot{ \gamma}'$. Since $\ker D_q X$ is one-dimensional we have 
		\begin{equation}\label{eq:tilde_N=pm_N}
			\dot{ \gamma}'={T}'_q=\pm T_q=\pm \dot\gamma(0).
		\end{equation}
		Let us define the functions $\lambda_\gamma,{{\lambda}_{\gamma'}}:(-\ve,\ve)\to { \R}$ as follows
		\begin{equation}
			\lambda_\gamma(t)=g(\dot\gamma(t),X)|_{\gamma(t)},\qquad {{\lambda}_{\gamma'}(t)}=g(\dot{\gamma}'(t),X)|_{\gamma'(t)}.
		\end{equation}
		Denoting with ${\mathcal C}'=\gamma'(-\ve,\ve)$, the assertion that the two functions have the same order at $t=0$ is equivalent to prove the following 
		\begin{equation}\label{eq:latter}
			\mathrm{order}(g({T}',X)|_{{\mathcal C}'},q)=\mathrm{order}(g(T,X)|_{\mathcal C},q).
		\end{equation}
		where we denote by $\mathrm{order}(f|_{{\mathcal C}},q)$ the order of vanishing of $f$ at $q$ along a curve $\mathcal C$ (cf. \cref{a:contact}).
	
		{ We now prove \eqref{eq:latter}. If both orders in
			\eqref{eq:latter} are infinite, the equality is immediate.
			Otherwise, we may assume that
			\[
			k:=\mathrm{order}\bigl(g(T,X)|_{\mathcal C},q\bigr)<+\infty.
			\]
			We shall prove that
			$\mathrm{order}(g(T',X)|_{\mathcal C'},q)=k$.} By \eqref{eq:C_omega_S} applied to $\mathcal C'$ one has
		\begin{equation}\label{eq:latter1}
			{\mathcal C}'=\{p\in S\mid \omega(T')|_p=0\}.
		\end{equation}
		Denote $\mathrm{order}(\mathcal C\cap{\mathcal C}',q)$ the order of contact between $\mathcal C$ and $\mathcal C'$ at $q$ (see \cref{def:order_of_contact} for a formal definition). By \eqref{eq:latter1}, applying \cref{lem:order=order_of_eq} we find
		\begin{equation}\label{eq:order_of_contact_mathcal}
			\begin{aligned}
				\mathrm{order}(\mathcal C\cap{\mathcal C}',q)+1&= \mathrm{order}\left(\omega(T')|_{\mathcal C},q\right)=\mathrm{order}\left(g(T,X)g(N,T')|_\mathcal C,q\right)\\
				&>\mathrm{order}\left(g(T,X)|_\mathcal C,q\right),
			\end{aligned}
		\end{equation}
		where in the second equality we have used the expression \eqref{eq:omega_S} of $\omega|_S$ and $g(N,X)|_\mathcal C\equiv 0$, thanks to \eqref{eq:C_omega_S}. The third inequality follows from $g(N,{T'})|_q=0$. 
		
		By construction $\omega(T')|_{{\mathcal C}'}\equiv 0$, hence by \eqref{eq:omega_S} 
		\begin{equation}\label{eq:det0}
			g(N,X)g(T,T')|_{{\mathcal C}'}=\left.{g(T,X)g(N,T')}\right|_{{\mathcal C}'}.
		\end{equation}
		Decomposing $X=g(T,X)T+g(N,X)N$ and using that $g(T,T')\neq 0$ around $q$ we get
		\begin{align}
			g(T',X)|_{{\mathcal C}'}&=\left.\left(g(T,X)g({T'},T)+g(N,X)g(T',N)\right)\right|_{{\mathcal C}'} \nonumber \\
			&=g(T,X)g({T'},T)\left.\left(1+ \frac{g(T',N)^{2}}{g(T,T')^{2}}\right)\right|_{{\mathcal C}'}\label{eq:g(tilde_N,X)}
		\end{align}
		where in the last equality we used \eqref{eq:det0}. Notice that on the right hand side of \eqref{eq:g(tilde_N,X)} the quantity  $g(T,X)$ is multiplied by two functions that are non vanishing at $q$ (notice $g(T',N)|_{q}=0$). Hence one has
		\begin{equation}\label{eq:change_function}
			\mathrm{order}(g(T',X)|_{{\mathcal C}'},q)=\mathrm{order}(g(T,X)|_{{\mathcal C}'},q)=\mathrm{order}(g(T,X)|_{\mathcal C},q),
		\end{equation}
		where the last equality combines \eqref{eq:order_of_contact_mathcal} and \cref{lem:order_lem1}. 
		
			{(ii)}. Finally, let us show that if $\lambda_\gamma$ has finite and odd order $k$ at $t=0$ then the quantity
    \begin{equation}\label{eq:k-th_eigenvalue}
        \Lambda^{(k)}(q):=\left.\frac{d^k}{dt^k}\lambda_\gamma(t)\right|_{t=0}
    \end{equation}
   is independent of the choice of the horizontal kernel extension $\gamma$.
   
   With the following computation we prove that, if the order $k$ of $\lambda_\gamma$ and ${{\lambda}_{\gamma'}}$ at $t=0$ is odd, then the leading order coefficients of the two functions coincide. Let us start by observing that, by \eqref{eq:g(tilde_N,X)},
		\begin{equation}
					\left.\frac{d^k}{dt^k}{\lambda_{\gamma'}}(t)\right|_{t=0}=\left.\frac{d^k}{dt^k}\right|_{t=0}g(T',X)|_{{\mathcal C}'}=\left.\frac{d^k}{dt^k}\right|_{t=0}\left [ g(T,X)g({T'},T)\left(1+ \frac{g(T',N)^{2}}{g(T,T')^{2}}\right)\right ].
		\end{equation}
		Using now that $\frac{d^j}{dt^j}\lambda_\gamma(t)\big|_{t=0}=\frac{d^j}{dt^j}\big|_{t=0}g(T',X)|_{{\mathcal C}'}=0$ for every $0\leq j \leq k-1$, one has
		\begin{align}
					\left.\frac{d^k}{dt^k}\right|_{t=0}g(T',X)|_{{\mathcal C}'}&=\left [ \left.\frac{d^k}{dt^k}\right|_{t=0} g(T,X)|_{{\mathcal C}'}\right ] g({T'},T)\left(1+ \frac{g(T',N)^{2}}{g(T,T')^{2}}\right)\\
					&=(\pm 1)\left [ \left.\frac{d^k}{dt^k}\right|_{t=0} g(T,X)|_{{\mathcal C}'}\right ] 
		\end{align}
		since $g({T'},T)|_{q}=\pm 1$ at the point $q$ and the term and $g(T',N)|_{q}=0$.
Using then $\dot\gamma(0)=(\pm 1)\dot{\gamma}'(0)$ (the sign being the same as $g({T'},T)|_{q}$), thanks to \eqref{eq:lemmachiave} in Remark~\ref{r:lemmachiave} one gets
	\begin{align}
					\left.\frac{d^k}{dt^k}\right|_{t=0}g(T',X)|_{{\mathcal C}'}=(\pm 1)\left [ \left.\frac{d^k}{dt^k}\right|_{t=0} g(T,X)|_{{\mathcal C}'}\right ] =(\pm 1)^{k+1}\left [ \left.\frac{d^k}{dt^k}\right|_{t=0} g(T,X)|_{{\mathcal C}}\right ] 
		\end{align}
 and the conclusion follows since $k$ is odd.
		\end{proof}
{\begin{remark}\label{rmk:Lambda_sign}
		The proof above shows that, when the order of degeneracy $k$
		is finite and even, $\lambda_\gamma^{(k)}(0)$ is independent
		of the horizontal kernel extension once its initial unit tangent
		is fixed, while reversing this tangent changes its sign.
		Since $\lambda_\gamma^{(k)}(0)$ is nonzero, it determines a canonical
		orientation of $\ker D_qX$: we declare $\dot\gamma(0)$ to be
		positively oriented if $\lambda_\gamma^{(k)}(0)>0$. Throughout the paper, at characteristic points with even order $k$ of degeneracy, unit vectors in
		$\ker D_qX$ and initial velocities of horizontal kernel
		extensions are chosen consistently with this orientation. In particular, for even $k$, we define
		\[
		\Lambda^{(k)}(q):=\bigl|\lambda_\gamma^{(k)}(0)\bigr|,
		\]
		whose value is independent of the horizontal kernel extension.
		With this convention, \eqref{eq:k-th_eigenvalue} holds also
		when $k$ is even.
\end{remark}}

    \section{Normal forms and coordinate characterization. Proof of Theorem~\ref{thm:finite_order_intro}}\label{ssec:order_main_thm}
 
    We provide a local coordinates characterization of the finite order condition. {Throughout the section, whenever the order of degeneracy of a given characteristic point is even, kernel extensions are intended to be oriented according to \cref{rmk:Lambda_sign}.}
    \begin{prop}\label{lem:normal_form_strong_body}
    Let $S$ be a surface embedded in a $3D$ contact sub-Riemannian manifold with characteristic vector field $X$. 
    Let $q\in S$ be a characteristic point with order of degeneracy  $\mathrm{ord}(q)=k\geq 2$. Then there exists coordinates $(x,y)$ on $S$ near $q$ such that 
\begin{equation}\label{eq:normal_form_strong}
                   q=(0,0),\qquad X=x(\nu+a_1(x,y))\partial_x+\left(\frac{\mu}{k!} y^k+xa_0(x,y)+b_0(y)\right)\partial_y,
               \end{equation}
               where $\nu=\mathrm{div}(X)|_q$,  $\mu=\Lambda^{(k)}(q)$ is defined in \eqref{eq:k-th_eigenvalue}, and $a_0,a_{1},b_0$ are smooth functions satisfying \begin{equation}\label{eq:needed_orders_app_body}
                b_0(y)=O(y^{k+1}),\qquad a_i(x,y)=O\left(\sqrt{x^2+y^2}\right),\qquad i=0,1.
               \end{equation}
               Moreover, in these coordinates, the curve $y\mapsto (0,y)$ is an horizontal kernel extension.
\end{prop}
\begin{proof}
    Let $q$ be a characteristic point such that $\det(D_qX)=0$. By \cref{r:tresei}, $\nu:=\mathrm{div}(X)|_q=\mathrm{trace}(D_qX)\neq 0$. Thus, the characteristic polynomial of $D_qX:T_qS\to T_qS$ reads
    \begin{equation} \label{eq:ipo1}
        P(\lambda)=\lambda(\lambda-\nu),
    \end{equation}
   and $D_qX$ is diagonalizable at $q$. In particular, there exists a smooth regular curve $\sigma:(-\ve,\ve)\to S$ such that 
    \begin{equation} \label{eq:ipo2}
    	\sigma(0)=q,\qquad D_qX(\dot\sigma(0))=\nu\dot\sigma(0).
    \end{equation}
 Let $\gamma$ be a horizontal kernel extension at $q$ and $T\in\mathfrak X(S)$ be a vector field extending $\dot \gamma$. Since  $\dot \gamma(0)$ and $\dot \sigma(0)$ are linearly independent, for $\ve>0$ small enough, the following map defines coordinates near $q$:
 \begin{equation} 
 	\varphi:(-\ve,\ve)^2\to S,\qquad \varphi(x,y)=e^{yT}\circ\sigma(x).
 \end{equation}
	In such coordinates $q=(0,0)$ and moreover
	\begin{equation}\label{eq:ipo3}
		\dot \sigma(0)= \partial_x,\qquad \gamma(t)=(0,t)\qquad \forall\,t \in(-\ve,\ve).
	\end{equation}
    Hence there exists $a,b:(-\ve,\ve)^2\to\mathbb R$ smooth functions such that 
    \begin{equation}
        X=(\nu x+a(x,y))\partial_x+b(x,y)\partial_y.
    \end{equation}
which, thanks to the assumptions on $D_{q}X$ in \eqref{eq:ipo1}, \eqref{eq:ipo2}, \eqref{eq:ipo3}, also  satisfy 
\begin{equation}\label{eq:a_b_orders}
    a(0,0)=b(0,0)=0,\qquad \nabla a(0,0)=\nabla b(0,0)=0.
\end{equation}
Moreover, for any $t\in(-\ve,\ve)$ the curve $\gamma(t)=(0,t)$ should be an integral curve of $X$, hence $a(0,y)=0$. Writing 
\begin{equation}
	a(x,y)=a(0,y)+x\int_{0}^1\partial_x a(tx,y)dt,
\end{equation}
it follows that there exists a smooth function $a_1(x,y)$ such that 
$
	a(x,y)=xa_1(x,y)
$.
Since $a(0,0)=0$ and $\nabla a(0,0)=0$, we deduce that $a_1(0,0)=0$,  thus, by differentiability of $a_1$ the function $a_{1}$ satisfies \eqref{eq:needed_orders_app_body}. Moreover, since $a(0,y)=0$ the function $\lambda_\gamma$ can be computed as
\begin{equation}
	\lambda_\gamma(t)=g(X|_{(0,t)},\dot\gamma(t))=b(0,t).
\end{equation}
Combining \eqref{eq:a_b_orders} and  the fact that we assume $\mathrm{ord}(q)=k$, then the function $b$ satisfies 
\begin{equation}\label{eq:orders_of_b}
	b(0,0)=0,\qquad \partial_xb(0,0)=0,\qquad \partial_y^jb(0,0)=\begin{cases}
		0,\quad 0\leq j<k,\\
		\mu,\quad j=k.
	\end{cases}
\end{equation}
Repeating the argument above, we can write for the function $b$  
\begin{equation}
	b(x,y)=b(0,y)+x\int_{0}^1\partial_xb(tx,y)dt=\frac{\mu}{k!}y^k+x\int_{0}^1\partial_xb(tx,y)dt+\left(b(0,y)-\frac{\mu}{k!}y^k\right).
\end{equation}
Now we define 
\begin{equation}
	a_0(x,y):=\int_{0}^1\partial_xb(tx,y)dt,\qquad b_0(y):=b(0,y)-\frac{\mu}{k!}y^k.
\end{equation}
By \eqref{eq:orders_of_b}, we have $a_0(0,0)=0$, thus, by differentiability of $a_0$, we obtain $a_0(x,y)=O(\sqrt{x^2+y^2})$. Furthermore, \eqref{eq:orders_of_b} implies that $b_0(y)=O(y^{k+1})$.
\end{proof}
\subsection{Proof of Theorem~\ref{thm:finite_order_intro}}
We first prove $(i)$ and $(iii)$. According to \cref{lem:normal_form_strong_body} there exist coordinates $(x,y)$ near  $q$ such that 
                \begin{equation}
                   X=x(\nu+a_1(x,y))\partial_x+\left(\frac{\mu}{k!}y^k+xa_0(x,y)+b_0(y)\right)\partial_y.
               \end{equation}
               with $a_{0},a_{1},b_{0}$ satisfying \eqref{eq:needed_orders_app_body}. For $s\in[0,1]$, we define $X_s$ as follows (notice that $X_1=X$)
                \begin{equation}\label{eq:explicit_X_s_body}
                   X_s=x(\nu+sa_1(x,y))\partial_x+\left(\frac{\mu}{k!}y^k+sxa_0(x,y)+sb_0(y)\right)\partial_y.
               \end{equation}
                We claim that there exists $\delta>0$ such that for all $s\in [0,1]$ the origin is the only critical point of $X_s$.
              
               Indeed, in a neighborhood of $q$, the singular points of $X_s$ are in bijection with the solutions of the system of equations
                \begin{equation}\label{eq:critical_system}
                    x(\nu+sa_1(x,y))=0, \qquad \frac{\mu}{k!}y^k+sxa_0(x,y)+sb_0(y)=0.
                \end{equation}
                Since $\nu\neq 0$ and $a_1(0,0)=0$, there exists $\delta>0$ such that, for $(x,y)\in B_\delta:=\{x^2+y^2\leq\delta^2\}$ and $s\in [0,1]$ the equation $x(\nu+sa_1(x,y))=0$ has a solution if and only if $x=0$. Substituting $x=0$ in the second equation of \eqref{eq:critical_system}, since $b_0(y)=O(y^{k+1})$, we find $y=0$. Therefore, the claim is proved. 
                
                It follows that the map
                \begin{equation}
                    F:[0,1]\times \partial B_\delta\to\mathbb S^1,\qquad F(s,p)=\left.\frac{X_s}{|X_s|}\right|_p,
                \end{equation}
                is a well defined smooth homotopy. Moreover $X_1=X$, therefore 
                \begin{equation}
                    \mathrm{ind}(q)=\deg\left(\left.\frac{X_0}{|X_0|}\right|_{\partial B_\delta}\right)=\deg\left(\left.\frac{X}{|X|}\right|_{\partial B_\delta}\right).
                \end{equation}
               For $s=0$, one has $X_0=\nu x\partial_x+\frac{\mu}{k!} y^k\partial_y$. According to \cref{lem:normal_form_strong_body}, $\nu=\mathrm{div}_q(X)$ and $\mu =\Lambda^{(k)}(q)$. Applying \cref{lem:index_lem} one gets $ \mathrm{ind}(q)=\mathrm{sign}(\nu\mu)$ for $k$ odd and $ \mathrm{ind}(q)=0$ for $k$ even,  which concludes the proof of $(iii)$.
                
                We now focus on item $(ii)$. 
                {Let $g_1,g_2$ be two Riemannian metrics defined in a compact neighborhood $U$ of $q$, and let $|\cdot|_{g_{i}}$ be the induced norms. Since all Riemannian metrics induce locally equivalent norms, we can find a constant $C\geq 1$ such that 
                \begin{equation}
                	C^{-1}|v|_{g_1}\leq |v|_{g_2}\leq C |v|_{g_1},\qquad \forall\,\, v\in TU.
                \end{equation}
                 Thus, it is sufficient to prove the estimate of point $(ii)$ for one particular Riemannian metric of our choice. We choose the metric which makes the fields $\partial_x,\partial_y$ an orthonormal frame. To simplify the notation, we may assume here  $\nu=\mu/k!=1$, as the exact values of these constants do not affect the validity of the estimate in question.} Therefore 
                \begin{equation}\label{eq:explicit_|X|^2}
                \begin{aligned}
                    |X|^2&=(x+xa_1(x,y))^2+\left(y^k+xa_0(x,y)+b_0(y)\right)^2.\\
                \end{aligned}
                \end{equation}
            	Notice that, according to \eqref{eq:needed_orders_app_body}, the function $f(y):=1+b_0(y)/y^k$ is smooth and satisfies
	            	\begin{equation}\label{eq:bbb}
            		f(y)=1+O(y).
            	\end{equation}
                Hence, we can compute 
                \begin{equation}
                	\begin{aligned}
                	|X|^2-(x^2+y^{2k})&=(x+xa_1(x,y))^2+\left(y^kf(y)+xa_0(x,y)\right)^2-(x^2+y^{2k})\\
                	&=x^2(a_1^2(x,y)+2a_1(x,y)+a_0^2(x,y))+y^{2k}(f^2(y)-1)+2xy^kf(y)a_0(x,y)\\
                	&=x^2\psi_{1}(x,y)+y^{2k}\phi(y)+2xy^k\psi_{2}(x,y),
                	\end{aligned}
                \end{equation}
Then we can write
                \begin{equation}
                \begin{aligned}
                    \left|\left(\frac{|X|^2}{x^2+y^{2k}}\right)-1\right|&\leq \left|\frac{x^2\psi_{1}(x,y)}{x^2+y^{2k}}\right|+\left|\frac{y^{2k}\phi(y)}{x^2+y^{2k}}\right|+\left|\frac{2xy^{k}\psi_{2}(x,y)}{x^2+y^{2k}}\right|
                    \leq |\psi_{1}(x,y)|+|\phi(y)|+|\psi_{2}(x,y)|.
                \end{aligned}
                \end{equation}
               Since $\psi_{i}(x,y)=O(\sqrt{x^{2}+y^{2}})$ for $i=1,2$ and $\phi(y)=O(y)$,  thanks to \eqref{eq:needed_orders_app_body} and \eqref{eq:bbb}, we deduce that 
                \begin{equation}
                    \lim_{(x,y)\to (0,0)}\left(\frac{|X|^2}{x^2+y^{2k}}\right)=1
                \end{equation}
                concluding the proof of $(ii)$.

            \section{Levi-Civita connection forms and Riemannian $\ve$-frames}\label{sec:ve-frames}
     
 In this section, we review the definition and fundamental properties of the Levi-Civita connection 1-forms, recalling their relation to the indices of isolated zeroes of vector fields.
       \subsection{Levi-Civita connection forms} \label{sec:levi_forms}

We recall that given a  smooth map $F:\mathbb S^1\to \mathbb S^1$ written as follows $F(e^{i\theta})=e^{if(\theta)}$,  where $f:\mathbb R\to\mathbb R$ is smooth, then the degree of $F$, denoted $\mathrm{deg}(F)$, is the following integer
                \begin{equation}\label{eq:grado}
                \mathrm{deg}(F)=\frac{1}{2\pi}(f(2\pi)-f(0))=\frac{1}{2\pi}\int_{0}^{2\pi}f'(\tau)d\tau.
                \end{equation}
           {
           	\begin{remark}\label{rmk:orientation}
           		Given a smooth oriented closed surface $S$ and a smooth compact submanifold with boundary $U\subset S$, we denote the boundary of $U$, oriented by the outward pointing normal, by $\partial U$. We denote the same boundary with the opposite orientation
           		by $-\partial U$. With this convention intended, for every smooth $1$-form $\eta$ on $S$, Stokes' theorem gives
           		\[
           		\int_U d\eta = \int_{\partial U}\eta,\qquad \int_{S\setminus U}d\eta
           		= \int_{-\partial U}\eta
           		= -\int_{\partial U}\eta.
           		\]
           	\end{remark}
           In what follows, a smooth ball $B$ in $S$ denotes the
           interior of a smoothly embedded compact disk.}
            \begin{defi}
                    Let $(S,g)$ be a smooth Riemannian surface, let $X\in\mathrm{Vec}(S)$ a smooth vector field and let $q\in S$ be  an isolated zero of $X$. Let $B$ be smooth ball such that $q$ is the unique zero of $X|_B$. Consider the following map 
                    \begin{equation}\label{eq:map_varphi_index}
                        F:\partial B\simeq \mathbb S^1\to \mathbb S^1,\qquad {F}(p)=\frac{X}{|X|}\bigg|_{p}.
                    \end{equation}
                    The index of the critical point $q$ is defined as the degree of the map $F$
                    \begin{equation}
                        \mathrm{ind}(q)=\deg(F).
                    \end{equation}
                \end{defi}

        We now define Levi-Civita connection form associated with a smooth vector field (playing the role of the characteristic vector field) and list some of its classical properties. 
        \begin{prop}\label{prop:levi-civita}
        Let $(S,g)$ be an oriented Riemannian surface and let $X$ be a smooth vector field vanishing on a set $Z\subset S$. On $S\setminus Z$ we define the orthonormal frame 
        \begin{equation} \label{eq:framedim}
            e_1:=\frac{X}{|X|},\qquad e_2:=\frac{JX}{|X|},
        \end{equation}
        where $J:TS\to TS$ is the complex structure of the Riemannian surface $(S,g)$. The Levi-Civita connection form associated with the frame $e_1,e_2$ is the following differential 1-form $\eta\in\Omega^1(S\setminus Z)$
       {\begin{equation}
            \eta(v)=g(\nabla_{v}e_2,e_1)=\frac{1}{|X|^2}g(\nabla_v JX,X),\qquad \forall\,v\in T\left(S\setminus Z\right).
        \end{equation}}
        The Levi-Civita connection form satisfies the following properties:
        \begin{enumerate}[label = $(\roman*)$]
        \item Let $\theta_1,\theta_2$ be the coframe dual to $e_1,e_2$, i.e. $\theta_i(e_j)=\delta_{ij}$. Then
        \begin{equation}\label{eq:levi_coords}
        \eta=-c_1\theta_1-c_2\theta_2,
        \end{equation}
        where $c_1,c_2$ are the structural functions satisfying  $d\theta_i=c_i\theta_1\wedge \theta_2$, for $i=1,2$.
        \item Let $K$ denote the Gaussian curvature and $\sigma$ the area form of $(S,g)$, then on $S\setminus Z$
            \begin{equation}\label{eq:levi_cartan}
                d\eta=K\sigma.
            \end{equation}
        {\item
        Assume that $q$ is an isolated zero of $X$, and let
        $\{B_\delta\}_{\delta>0}$ be a family of smooth balls
        in $S$ containing $q$ and satisfying
        $\operatorname{diam}_g(B_\delta)\to 0$ as $\delta\to 0$.
        Then
        \begin{equation}\label{eq:levi_int_ind}
        	\lim_{\delta\to 0}\int_{\partial B_\delta}\eta
        	= -2\pi\operatorname{ind}(q).
        \end{equation}}
        \end{enumerate}
        \end{prop}
        \begin{proof}
              $(i),(ii)$. These are standard properties of the Levi-Civita form. See for instance \cite[Ex.\,3.4.28]{Petersen} or \cite[Sec.~4.4.1]{Agrachev} for details. For completeness we prove $(i)$.

                 The forms $\theta_1,\theta_2$ form a basis of 1-forms on $S\setminus Z$, therefore $\eta=\eta(e_1)\theta_1+\eta(e_2)\theta_2.$
                 We can then compute the coefficient
            {\begin{equation}
            \begin{aligned}
                \eta(e_1)&=g(\nabla_{e_1}e_2,e_1)=g(\nabla_{e_2}e_1+[e_1,e_2],e_1)=g([e_1,e_2],e_1)=-c_1.
                \end{aligned}
                \end{equation}}
               Where we used the fact that $\nabla$ is metric and torsion free, and $g(e_{i},e_{j})$ is constant. Repeating analogous computations  one obtains $\eta(e_2)=-c_2$.

                $(iii)$. If $\delta>0$ is small enough we may assume that $q$ is the only zero of $X|_{B_\delta}$. We would like to express  in coordinates 
                \begin{equation}\label{eq:varphi_index}
                    F:\partial B_\delta\to \mathbb S^1,\qquad F(p)=\frac{X}{|X|}\bigg|_{p}=e_{1}|_{p}.
                \end{equation}
               {Fix a regular orientation preserving parametrization 
               $
                    \gamma:[0,2\pi]\to  \partial B_\delta
              $.}
                Choosing $\delta>0$ small enough, consider an oriented smooth orthonormal frame $e_1',e_2'$ on $B_\delta$ and a function $\phi:[0,2\pi]\to \mathbb R$ such that
                \begin{equation}\label{eq:map_varphi}
                    F(\gamma(t))=\cos\phi(t)e_1'+\sin\phi(t)e_2',\qquad \forall\,t\in [0,2\pi].
                \end{equation}
                Moreover, by \eqref{eq:grado}, one has 
                \begin{equation}
                    \phi(2\pi)-\phi(0)=2\pi\deg(F)=2\pi\,\mathrm{ind}(q).
                \end{equation}
                Consider the Levi-Civita connection form $\eta'$ associated with the frame $e_1',e_2'$, namely
                {\begin{equation}
                    \eta'(v)=g(\nabla_v e_2',e_1'),\qquad \forall\,v\in TB_\delta.
                \end{equation}}
            Note that due to \eqref{eq:varphi_index} and \eqref{eq:framedim}, along $\gamma(t)$, one has
                $$e_{1}'=\cos\phi(t)\, e_1-\sin\phi(t) \,e_2,\qquad e_{2}'=\sin\phi(t) \,e_1+\cos\phi(t)\,e_2$$
                 We stress that $e_1',e_2'$ is smooth, while $e_1,e_2$ is  singular on $Z$. For any $t\in[0,2\pi]$ we can compute 
                {\begin{equation}\label{eq:eta_difference}
                \begin{aligned}
                    \eta'(\dot\gamma(t))&=g(\nabla_{\dot\gamma(t)} e_2',e_1')=g(\nabla_{\dot\gamma(t)} (\sin\phi(t) \,e_1+\cos\phi(t)\,e_2),\cos\phi(t)\, e_1-\sin\phi(t) \,e_2)\\
                    &=\dot \phi(t)+g(\nabla_{\dot\gamma(t)} e_2,e_1)=\dot \phi(t)+\eta(\dot\gamma(t)).
                \end{aligned}
                \end{equation}}
                Notice that, since $\eta'$ is smooth on $B_\delta$, applying Stokes theorem we find by absolute continuity of the integral
                \begin{equation}\label{eq:acoi}
                    \lim_{\delta\to 0}\int_{\partial B_\delta}\eta'=\lim_{\delta\to 0}\int_{B_\delta}d\eta'=0.
                \end{equation}
                We next  combine \eqref{eq:acoi} together with \eqref{eq:eta_difference} to compute the  limit
                \begin{equation}
                \begin{aligned}
                    \lim_{\delta\to 0}\int_{\partial B_\delta}\eta&=\lim_{\delta\to 0}\int_{\partial B_\delta}\eta-\eta'=\lim_{\delta\to 0}\int_{0}^{2\pi}\left(\eta(\dot\gamma(t))-\eta'(\dot\gamma(t))\right)dt\\
                    &{= -\lim_{\delta\to 0}\int_{0}^{2\pi}\dot {\phi}(t)dt=\phi(0)-\phi(2\pi)=-2\pi\,\mathrm{ind}(q).} \hfill  \qedhere
                    \end{aligned} 
                \end{equation}
        \end{proof}
            \subsection{Riemannian $\ve$-frames}\label{sec:ve_frame}
            
            We apply the construction of Section~\ref{sec:levi_forms} to the case when $X$ is the characteristic vector field, hence $Z=\Sigma(S)$ the characteristic set, and the metrics are given by the Riemannian approximations $g^\ve$ inducing Riemannian structures  $(S,g^\ve|_S)$.      
 We denote by $|\cdot|_\ve$  the norm induced by $g^\ve$ and $J^\ve:TS\to TS$ is the corresponding complex structure.
            \begin{defi}\label{defi:ve-Riemannian_frames}
            Let $(M,\mathcal D,g)$ be a contact sub-Riemannian manifold and $S$ an embedded oriented surface with characteristic vector field $X$.  We define the Riemannian $\ve$-frame on $S\setminus \Sigma(S)$
            \begin{equation}
                e_1^\ve=\frac{X}{|X|_{\ve}},\qquad  e_2^\ve=\frac{J^\ve X}{|X|_{\ve}},
            \end{equation}
             We also denote the associated dual basis of 1-forms  $\theta_1^\ve,\theta_2^\ve$  on $S\setminus \Sigma(S)$ satisfying
            \begin{equation}
              \theta_i^\ve(e_j^\ve)=\delta_{ij},\qquad i,j=1,2.
            \end{equation}
                \end{defi}
            For simplicity of notation, when $\ve=1$ we will omit superscripts
            \begin{equation}\label{eq:1-Riemannian_framing}
                e_i:=e_i^1,\qquad \theta_i:=\theta_i^1,\qquad i=1,2.
            \end{equation}
        \begin{remark}\label{rmk:theta_2}
        	By definition \eqref{eq:def_X} of characteristic vector field $X$, outside the characteristic set  $\ker\theta_2=\ker\omega|_S$, where $\omega$ is the normalized contact form. Therefore there exists a function $f$ smooth  outside of the characteristic set, such that 
        	\begin{equation}
        		\omega=f\theta_2.
        	\end{equation}
        	Let $\sigma$ be the area form of $g^1|_S$. Using $\theta_2(e_2)=1$  and $\sigma(e_1,e_2)=1$, we may compute the function $f$
        	\begin{equation}
        		f=\omega(e_2)=\sigma(X,e_2)=\sigma(|X|e_1,e_2)=|X|.
        	\end{equation}
        \end{remark}
         \begin{lemma}
            Let $(M,\mathcal D,g)$ be a contact sub-Riemannian manifold and $S$ an embedded oriented surface with characteristic vector field $X$. For any $\ve>0$ let $b_\ve:S\to \mathbb R$ denote the following function 
            \begin{equation}\label{eq:b_ve}
                b_\ve=\sqrt{1-\mathrm{div}(X)^2(1-\ve)}=\sqrt{1-(1-|X|^2)(1-\ve)}.
            \end{equation}
            The Riemannian $\ve$-coframe of \cref{defi:ve-Riemannian_frames} satisfies
            \begin{equation} \label{eq:dual_perturbed}
                \theta_1^\ve=\theta_1,\qquad \theta_2^\ve=\frac{b_\ve}{\sqrt{\ve}}\theta_2.
            \end{equation}
        \end{lemma}
        \begin{proof}
        We denote with $f_1=X/|X|$ the normalized characteristic vector field of $S$, with $f_2$ the oriented horizontal orthonormal, and with $f_0$ the Reeb vector field. Denoting with $\nu_1,\nu_2,\nu_0$ the corresponding dual coframe, i.e. $\nu_i(f_j)=\delta_{ij}$, according to \eqref{eq:g_ve_properties} the Riemannian metric $g^\ve$ reads
        \begin{equation}\label{eq:g_ve_nu_i}
            g_\ve=\nu_1\otimes\nu_1+\nu_2
            \otimes\nu_2+\frac{1}{\ve}\nu_0\otimes\nu_0.
        \end{equation}
        Let $e_1,e_2$ be the Riemannian frame of equation \eqref{eq:1-Riemannian_framing}, i.e., $e_1=f_1$ and $e_2$ is the oriented orthonormal complement obtained with respect to the metric $g^1|_S$. Substituting $\ve=1$ in the expression \eqref{eq:g_ve_nu_i} we deduce the existence of a smooth function $a:S\setminus{\Sigma(S)}\to [-1,1]$ such that 
        \begin{equation}\label{eq:frame_ve=1}
            e_1=f_1=\frac{X}{|X|},\qquad e_2=af_2+\sqrt{1-a^2}f_0.
        \end{equation}
        Denoting with $\sigma$ the area form induced by $g^1$ on $S$, since $\sigma(e_1,e_2)=1$ we can compute
        \begin{equation}
                \begin{aligned}
                |X|^2&=|X|^2\sigma(e_1,e_2)=|X|\sigma(X,e_2)=|X|\iota_X\sigma(e_2)=|X|\omega(e_2)=|X|\sqrt{1-a^2},
                \end{aligned}
            \end{equation}
        where in the fourth equality we have used the definition of the characteristic vector field \eqref{eq:def_X}, and in the last one the fact that $\omega(f_2)=0$ and $\omega(f_0)=1$. We deduce that
        \begin{equation}\label{eq:computation_of_a}
            a^2=1-|X|^2=\mathrm{div}(X)^2,
        \end{equation}
        where the last equality is obtained from equation \eqref{eq:div(X)^2}. Note that
            \begin{equation}
                g^\ve(e_1,e_1)=1,\qquad g^\ve(e_2,e_1)=0,\qquad g^\ve(e_2,e_2)=a^2+\frac{1}{\ve}(1-a^2)=\frac{b_\ve^2}{\ve},
            \end{equation}
            therefore, the Riemannian $\ve$-frame of \cref{defi:ve-Riemannian_frames} can be computed as
            \begin{equation}\label{eq:controvariant_ve}
                e_1^\ve=e_1,\qquad e_2^\ve=\frac{\sqrt{\ve}}{b_\ve}e_2,
            \end{equation}
           and, consequently, the associate coframe has the expression  \eqref{eq:dual_perturbed}.
        \end{proof}
      
  \subsection{Estimates on Levi-Civita forms}

            Given an oriented Riemannian surface $(S,g)$ and a differential form $\eta\in \Omega^1(S)$ we define the norm of $\eta$ at the point $q\in S$ as
        \begin{equation}
            \|\eta\|(q)=\sup_{v\in T_qS,\,v\neq 0}\frac{|\eta_q(v)|}{|v|},
        \end{equation}
        where $|v|=\sqrt{g(v,v)}$. In particular, the norm of a differential 1-form is a function 
       $            \|\eta\|:S\to\mathbb R
        $. 
       
        Observe that for $\alpha,\beta\in\Omega^1(S)$, and $\sigma$ is the area form of $(S,g)$, the following inequality holds 
        \begin{equation}\label{eq:determinant_inequality}
            \left|\int_S\alpha\wedge\beta\,\right|\leq \int_S\|\alpha\|\|\beta\|\sigma.
        \end{equation}

          We now compute the Levi-Civita connection form $\eta^\ve$ associated with a Riemannian $\ve$-frame of \cref{defi:ve-Riemannian_frames}, and analyse some of its properties.
      \begin{lemma}\label{lem:levi_ve0}
             Let $\eta^\ve$ be the Levi-Civita connection 1-form  associated with the Riemannian $\ve$-frame. Then
            \begin{equation}\label{eq:levi_ve}
                \eta^\ve=-\frac{\sqrt{\ve}}{b_\ve}c_1\theta_1-c_2\frac{b_\ve}{\sqrt{\ve}}\theta_2-\frac{e_1(b_\ve)}{\sqrt{\ve}}\theta_2,
            \end{equation}
            where $c_1,c_2\in C^{\infty}(S\setminus \Sigma(S))$ satisfies $d\theta_i =c_i\, \theta_1\wedge \theta_2$, $i=1,2$, and $b_\ve$ is  defined in \eqref{eq:b_ve}.  
        \end{lemma}
        \begin{proof}
 We compute the differentials of the coframe $\theta_1^\ve, \theta_2^\ve$  in \eqref{eq:dual_perturbed}:
        \begin{equation}\label{eq:c1_per}
            d\theta_1^\ve=c_1\theta_1\wedge\theta_2=\frac{\sqrt{\ve}}{b_\ve}c_1\theta_1^\ve\wedge\theta_2^\ve,
        \end{equation}
        Analogously for $\theta_2^\ve=(b_\ve/\sqrt{\ve})\theta_2$ we get:
        \begin{equation}\label{eq:c2_per}
            d\theta_2^\ve=\frac{1}{\sqrt{\ve}}e_1(b_\ve)\theta_1\wedge\theta_2+\frac{b_\ve}{\sqrt{\ve}}c_2\theta_1\wedge\theta_2=\left(\frac{e_1(b_\ve)}{b_\ve}+c_2\right)\theta_1^\ve\wedge\theta_2^\ve.
        \end{equation}
        According to \eqref{eq:levi_coords}, the Levi-Civita connection form associated with the frame $e_1^\ve,e_2^\ve$ reads
        \begin{equation}
            \eta_\ve=-c_1^\ve\theta_1^\ve-c_2^\ve\theta_2^\ve,\qquad c_{1}^{\ve}=\frac{\sqrt{\ve}}{b_\ve}c_1, \quad c_{2}^{\ve}=\left(\frac{e_1(b_\ve)}{b_\ve}+c_2\right).
        \end{equation}
        Substituting the expression for $\theta_1^\ve,\theta_2^\ve$ computed in equation \eqref{eq:dual_perturbed}, we get \eqref{eq:levi_ve}. 
     \end{proof}
     We are now ready to prove the following crucial estimates. 
        \begin{prop}\label{lem:levi_ve}
        Let $\eta^\ve$ be the Levi-Civita connection 1-form  associated with the Riemannian $\ve$-frame. The following assertions hold:
            \begin{itemize}
            \item[(i)] the form $\alpha\in\Omega^1(S\setminus \Sigma(S))$ defined by 
            \begin{equation}\label{eq:alpha}
                \alpha:=\lim_{\ve\to 0}\sqrt\ve\, \eta^\ve=-\frac{\mathrm{div}(X)}{|X|} \omega|_{S},
            \end{equation}
             is smooth on $S\setminus \Sigma(S)$ and uniformly bounded in norm. Furthermore
            \begin{equation}\label{eq:lim_diff=zero}
                \lim_{\ve\to 0}\left.\left(\eta^\ve-\frac{\alpha}{\sqrt\ve}\right)\right|_q=0,\qquad \forall\,q\in S\setminus\Sigma(S).
            \end{equation}
            \item[(ii)] There exists a positive constant $C>0$, independent of $\ve$, such that
            \begin{equation}\label{eq:dom_conv}
                 \left \|d\alpha\right\|\leq \frac{C}{|X|},\qquad \left \|\eta_\ve-\frac{\alpha}{\sqrt\ve}\right\|\leq \frac{C}{|X|},
            \end{equation}
            where the norms are computed with respect to the metric $g^1|_S$ of \cref{def:g_ve}.
            \end{itemize}

        \end{prop}
        \begin{proof} 

$(i)$. Notice that the function $(q,\ve)\mapsto b_\ve(q)$, defined in equation \eqref{eq:b_ve}, is strictly positive outside of the characteristic set, and smooth in both variables $(q,\ve)\in (S\setminus \Sigma(S))\times [0,1]$. Therefore, for any non characteristic point $q\in S$ we have 
        \begin{equation}\label{eq:lim_alpha}
            \lim_{\ve\to 0}\sqrt\ve \,\eta^\ve_q=-\lim_{\ve\to 0}\left.\left(\frac{\ve}{b_\ve}c_1\theta_1+c_2b_\ve\theta_2+e_1(b_\ve)\theta_2\right)\right|_q=-\left.\left(c_2b_0\theta_2+e_1(b_0)\theta_2\right)\right|_q=\alpha_q.
        \end{equation}
        Let us compute the coefficient $c_{2}$ satisfying $d\theta_2=c_2\theta_1\wedge\theta_2$. Writing $\theta_2=\iota_{e_1}\theta_1\wedge\theta_2$ and recalling the definition the divergence \eqref{eq:div(X)_formula0} one has
        \begin{equation}\label{eq:c_2_computation}
        \begin{aligned}
            c_2&=d\theta_2(e_1,e_2)=d\left(\iota_{e_1}\theta_1\wedge\theta_2\right)(e_1,e_2)=\mathrm{div}(e_1)=\mathrm{div}\left(\frac{X}{|X|}\right)\\
            &=X(|X|^{-1})+|X|^{-1}\mathrm{div}(X)=-|X|^{-3}g(\nabla_XX,X)+|X|^{-1}\mathrm{div}(X)\\
            &=-|X|^{-1}g(\nabla_{e_1}X,e_1)+|X|^{-1}\mathrm{div}(X),
        \end{aligned}
        \end{equation}
   where we used the formula $\mathrm{div}(fY)=Y(f)+f\mathrm{div}(Y)$, which holds for every smooth function $f$ and vector field $Y$.  By definition of $b_\ve$ in \eqref{eq:b_ve} we have $b_0=|X|$, therefore 
        \begin{equation}\label{eq:|X|c2}
            b_0c_2=|X|c_2=-g(\nabla_{e_1}X,e_1)+\mathrm{div}(X).
        \end{equation}
        Next we compute the term $e_1(b_0)$:
        \begin{equation}\label{eq:e1(b0)}
            e_1(b_0)=|X|^{-1}X(|X|)=|X|^{-2}g(\nabla_XX,X)=g(\nabla_{e_1}X,e_1).
        \end{equation}
        Combining equations \eqref{eq:lim_alpha}, \eqref{eq:|X|c2} and \eqref{eq:e1(b0)} and $\theta_{2}=\omega/|X|$ (cf.\ \cref{rmk:theta_2}), we deduce
        \begin{equation}
            \alpha_q=-(c_2b_0+e_1(b_0))\theta_2|_q=-\mathrm{div}(X)\,\theta_2|_q=-\frac{\mathrm{div}(X)}{|X|}\omega|_{q}.
        \end{equation}
        Since $\|\theta_2\|=1$ and $\mathrm{div}(X)$ is smooth  and well defined on the whole $S$, we deduce that $\alpha$ is uniformly bounded in norm.
        To prove \eqref{eq:lim_diff=zero}, we fix a non characteristic point $q\in S\setminus \Sigma(S)$ and consider the map \eqref{eq:b_ve} as a function of $\ve$:
            \begin{equation}\label{eq:bve(ve)}
              [0,1]\ni\ve\mapsto b_\ve(q)={\sqrt{1-{(1-|X_q|^2)(1-\ve)}}}.
            \end{equation}
            Since $q\in S\setminus \Sigma(S)$, then $|X_q|>0$. According to \cref{prop:div=pm}, $|X_q|\leq 1$.Therefore the map \eqref{eq:bve(ve)} is smooth and strictly positive. Therefore, the map
            \begin{equation}
               [0,1]\ni\ve\mapsto r_\ve(q):=\frac{b_\ve(q)-b_0(q)}{\ve},
            \end{equation}
            is also smooth. Consequently we have 
            \begin{equation}
            \begin{aligned}
                \left\|\eta^\ve-\frac{\alpha}{\sqrt{\ve}}\right\|&=\left\|\frac{\sqrt\ve}{b_\ve(q)}c_1\theta_1+\frac{c_2}{\sqrt\ve}(b_\ve(q)-b_0(q))\theta_2+\frac{1}{\sqrt\ve}e_1(b_\ve(q)-b_0(q))\theta_2\right\|\\
                &=\left\|\frac{\sqrt\ve}{b_\ve(q)}c_1\theta_1+\sqrt\ve c_2r_\ve(q)\theta_2+\sqrt\ve e_1(r_\ve)(q)\theta_2\right\|\to 0,\qquad \text{as}\,\,\, \ve\to 0,
            \end{aligned}
            \end{equation}
            where we used the fact that $\theta_1,\theta_2$ have constant norm $1$ with respect to the metric $g^1|_S$, and the fact that $b_\ve(q)$ is strictly positive and that $r_\ve(q)$, $e_1(r_\ve)(q)$ are smooth.
        
       $(ii)$. Let us compute $d\alpha$. Applying Leibnitz rule and  $d\theta_2=c_2\theta_1\wedge\theta_2$, we get
        \begin{equation}
        \begin{aligned}
            d\alpha&=-d\left(\mathrm{div}(X)\right)\wedge\theta_2-\mathrm{div}(X)d\theta_2=-(e_1(\mathrm{div}(X))+{\mathrm{div}(X)}c_2)\theta_1\wedge \theta_2\\
            &=-\left(e_1(\mathrm{div}(X))+{\mathrm{div}(X)}\left(\frac{\mathrm{div}(X)-g(\nabla_{e_1}X,e_1)}{|X|}\right)\right)\theta_1\wedge \theta_2,
        \end{aligned}
        \end{equation}
        where, in the third equality, we used the expression of $c_{2}$ given in \eqref{eq:c_2_computation}. Since $\mathrm{div}(X)$ is smooth, $|\mathrm{div}(X)|$ is bounded. Moreover, denoting $a=\mathrm{div}(X)$ and $\nabla a$ the gradient of $a$, we find 
        \begin{equation}\label{eq:e1(a)_est}
        	|e_1(a)|=|g(\nabla a, e_1)|\leq |\nabla a|.
        \end{equation}
    	Hence the term $|e_1(\mathrm{div}(X))|$ is uniformly bounded. Furthermore,  using that the map $v\mapsto \nabla_v X$ is a smooth endomorphism of $TS$ and $S$ is compact, we deduce the existence of a constant $C$ such that 
        \begin{equation}\label{eq:g_nabla_est}
        	|g(\nabla_{v}X,w)|\leq |\nabla_v X||w|\leq C|v||w|,\qquad \forall v,w\in TS.
        \end{equation}
         We deduce that $|g(\nabla_{e_1}X,e_1)|\leq C$. Therefore the first estimate in \eqref{eq:dom_conv} is proved.
        
        Moving to the second estimate, to lighten the notation we denote $a^2=\mathrm{div}(X)^2=1-|X|^2$, recalling \eqref{eq:div(X)^2}. Then, according to \eqref{eq:b_ve}
        \begin{equation}\label{eq:b_ve_a}
            b_\ve=\sqrt{1-a^2(1-\ve)}.
        \end{equation}
  		
        Observe that $a^2\in[0,1]$, hence $\sqrt{1-a^2(1-\ve)}-\sqrt{1-a^2}=|b_\ve-b_0|$ is an increasing function of $a$ for $a,\ve\in[0,1]$. Furthermore, $b_\ve\geq \sqrt\ve$. Thus one gets
        \begin{equation}\label{eq:bieps}
           |b_\ve-b_0|\leq\sqrt\ve\leq b_\ve.
        \end{equation}
    
        Combining \eqref{eq:bieps} with $\|\theta_i\|=1$ we deduce the following estimate 
            \begin{equation}\label{eq:est1}
            \begin{aligned}
                \left\|\eta^\ve-\frac{\alpha}{\sqrt{\ve}}\right\|&=\left\|\frac{\sqrt\ve}{b_\ve}c_1\theta_1+\frac{c_2}{\sqrt\ve}(b_\ve-b_0)\theta_2+\frac{1}{\sqrt\ve}e_1(b_\ve-b_0)\theta_2\right\|\\
                &\leq |c_1|+|c_2|+\frac{1}{\sqrt\ve}|e_1(b_\ve-b_0)|.
            \end{aligned}
            \end{equation}
        	The derivative $e_1(b_\ve)$ can be computed as follows 
        	\begin{equation}
        		e_1(b_\ve)=e_1\left(\sqrt{1-a^2(1-\ve)}\right)=\frac{-a e_1(a)(1-\ve)}{\sqrt{1-a^2(1-\ve)}}=\frac{-a e_1(a)}{b_\ve}(1-\ve).
        	\end{equation}
        	The function $e_1(b_0)$ is obtained substituting $\ve=0$ in the above expression, thus $e_1(b_0)=-ae_1(a)/b_0$. Thus we have
	\begin{equation}\label{eq:manipolo}
	  \frac{1}{\sqrt\ve}|e_1(b_\ve-b_0)|=\frac{1}{\sqrt\ve}\left |a e_1(a)\left(\frac{1-\ve}{b_\ve}-\frac{1}{b_0}\right)\right|
                =\frac{1}{\sqrt\ve}\left |b_{0}e_1(b_{0})\left(\frac{1-\ve}{b_\ve}-\frac{1}{b_0}\right)\right|
	\end{equation}
	Manipulating \eqref{eq:manipolo}, we can estimate
            \begin{equation}\label{eq:est2}
            \begin{aligned}
                \frac{1}{\sqrt\ve}|e_1(b_\ve-b_0)|
                &
                =\frac{1}{\sqrt\ve}\left |b_{0}e_1(b_{0})\left(\frac{b_0-b_\ve}{b_\ve b_0}-\frac{\ve}{b_\ve}\right)\right|\\
                &
                \leq\left|\frac{e_1(b_0)}{b_\ve}\frac{b_0-b_\ve}{\sqrt\ve}\right|+\left|\frac{b_{0}e_1(b_{0})\sqrt{\ve}}{b_\ve}\right|\\
                &\leq \left|\frac{e_1(b_0)}{b_0}\right|+|b_{0}e_1(b_{0})|,
            \end{aligned}
            \end{equation}
        	and in the  last inequality we have used the estimates \eqref{eq:bieps} and $b_\ve\geq b_0$.
			Substituting \eqref{eq:est2} into \eqref{eq:est1}
            \begin{equation}\label{eq:found_L^1_bound}
            \begin{aligned}
                \left\|\eta^\ve-\frac{\alpha}{\sqrt{\ve}}\right\|\leq|c_1|+|c_2|+|b_{0}e_1(b_{0})|+\frac{|e_1(b_0)|}{|X|}.
            \end{aligned}
            \end{equation}
            The function $a=\mathrm{div}(X)$ is smooth, by \eqref{eq:e1(a)_est} the term $|b_{0}e_1(b_{0})|=|a e_1(a)|$ is bounded. The term $e_1(b_0)$ is computed in \eqref{eq:e1(b0)} and, by \eqref{eq:g_nabla_est}, it is bounded. Furthermore, from \eqref{eq:|X|c2}, by smoothness of $\mathrm{div}(X)$ and \eqref{eq:g_nabla_est}, we deduce that 
            \begin{equation}\label{eq:c2}
                |c_2|=\frac{|-g(\nabla_{e_1}X,e_1)+\mathrm{div}(X)|}{|X|}\leq \frac{C}{|X|},
            \end{equation}
           	for some constant $C>0$. A computation analogous to that of equation \eqref{eq:c_2_computation} shows that 
             	\begin{equation}
        		\begin{aligned}
        			c_1&=d\theta_1(e_1,e_2)=-d\left(\iota_{e_2}\theta_1\wedge\theta_2\right)(e_1,e_2)=-\mathrm{div}(e_2)=-\mathrm{div}\left(\frac{JX}{|X|}\right)\\
        			&=-JX(|X|^{-1})-|X|^{-1}\mathrm{div}(JX)=|X|^{-3}g(\nabla_{JX}X,X)-|X|^{-1}\mathrm{div}({JX})\\
        			&=|X|^{-1}g(\nabla_{e_2}X,e_1)-|X|^{-1}\mathrm{div}(JX).
        		\end{aligned}
        	\end{equation}
			By smoothness of $\mathrm{div}(JX)$ and \eqref{eq:g_nabla_est}, we deduce that 
            \begin{equation}\label{eq:c1}
                |c_1|=\frac{|g(\nabla_{e_2}X,e_1)-\mathrm{div}(JX)|}{|X|}\leq \frac{C}{|X|}.
            \end{equation}
            Substituting estimates \eqref{eq:c2} and \eqref{eq:c1} in \eqref{eq:found_L^1_bound} we deduce that there exists a constant $C>0$ such that 
            \begin{equation}
                \left\|\eta^\ve-\frac{\alpha}{\sqrt{\ve}}\right\|\leq \frac{C}{|X|}.   
            \end{equation}
            which concludes the proof.
        \end{proof}
        We conclude the section with a technical lemma which is crucial in the proof of the main theorem.

\begin{lemma}\label{lem:technical_eta_ve}
        Let $q\in S$ be a characteristic point of finite order $k\in\mathbb N$ and $\ve \in (0,1]$. Then there exists a family of smooth curves $\{C_\delta\}_{\delta>0}$, enclosing the point $q$ and converging uniformly to it, such that
        \begin{equation}
        \lim_{\delta\to 0}\mathrm{length}\left(C_\delta\right)=0,\qquad  \limsup_{\delta\to 0}\int_{C_{\delta}}|\eta^\ve|<+\infty,
        \end{equation}
        where $\eta^\ve$ is the Levi-Civita connection form described \cref{lem:levi_ve}.
        \end{lemma}
   
        \begin{proof}
{We first assume $k\geq2$.} According to \cref{lem:normal_form_strong_body}, there exist coordinates $(x,y)$ near $q$ and $\mu,\nu\in\mathbb R\setminus \{0\}$ such that
\begin{equation}
q=(0,0),\qquad X=x(\nu+a_1(x,y))\partial_x+\left(\mu \frac{y^k}{k!}+xa_0(x,y)+b_0(y)\right)\partial_y,
\end{equation}
We set $C_\delta=\{(x,y)\in S\mid x^2+y^2=\delta^2\}$. Let $B_\delta=\{(x,y)\in\mathbb R^2\mid x^2+y^2 <\delta^2\}$. We may assume that the coordinates $(x,y)$ are defined on an open set containing the closure of $B_1$. Let $e_1',e_2'$ be a smooth oriented orthonormal frame for the metric $g^\ve|_S$ (in this proof $\ve\in (0,1]$ is fixed hence we omit it in the notation) defined on the whole $B_1$, and write
\begin{equation}\label{def:X'}
X=|X|(\cos\phi \, e_1'+\sin\phi \,e_2')=:{ V_1'e_1'+V_2'e_2'}
\end{equation}
for some smooth function $\phi:B_\delta\setminus\{q\}\to \mathbb R/({2\pi \mathbb Z})$. We parametrize $C_\delta$ by $\gamma:[0,2\pi]\to S$ as follows
\begin{equation}\label{eq:def_gamma}
\gamma(t)=(\delta\cos(t),\delta\sin(t)),\qquad \forall\,t\in[0,2\pi].
\end{equation}
For any function $f:U\to \mathbb R$ we denote the derivative along $\gamma$ with $\dot f$, in particular  
\begin{equation}\label{eq:der_along_gamma}
	\dot f = \frac{d}{dt} f(\gamma(t))=(x\partial_y f-y\partial_x f)|_{\gamma(t)}.
\end{equation}
Repeating verbatim the computations \eqref{eq:eta_difference}, which was needed in the proof of \cref{prop:levi-civita}, we find
{\begin{equation}\label{eq:eta_ve_dot_gamma1}
\eta^\ve(\dot\gamma)-g(\nabla_{\dot\gamma} e_2',e_1')=-\dot\phi.
\end{equation}}
Since ${ V_1'}=|X|\cos\phi$, ${ V_2'}=|X|\sin\phi$, we have $\left({ V_1'}\right)^2+{\left({ V_2'}\right)}^2= |X|^2$ and we can compute 
\begin{equation}\label{eq:dett}
\dot\phi=\frac{d}{dt}\left(\arctan{\left(\frac{{ V_2'}}{{ V_1'}}\right)}\bigg|_{\gamma(t)}\right)=\frac{{ V_1'}{\dot V_2'}-{ V_2'}{\dot V_1'}}{\left({ V_1'}\right)^2+\left({ V_2'}\right)^2}
=\frac{1}{|X|^2}{ V'} \wedge { \dot V'}
\end{equation}
where we identify ${ \dot V'}=({ \dot V_1'},{ \dot V_2'})$ as a vector in $\R^{2}$ and, given $v,w\in\mathbb R^2$, we set $v\wedge w:=\det(v,w)$. Combining \eqref{eq:dett} with \eqref{eq:eta_ve_dot_gamma1} yields
{\begin{equation}\label{eq:eta_ve_dot_gamma}
\eta^\ve(\dot\gamma)-g(\nabla_{\dot\gamma} e_2',e_1')=-\frac{1}{|X|^2}
{ V'} \wedge { \dot V'}
\end{equation}}
We now introduce the following notation
\begin{equation}
V_1=\nu x+xa_1(x,y),\qquad V_2=\mu \frac{y^k}{k!}+xa_0(x,y)+b_0(y),\qquad { V=(V_1,V_2)},
\end{equation}
so that $X=V_1\partial_x+V_2\partial_y$. Since both $e_1',e_2'$ and $\partial_x,\partial_y$ are smooth basis for $TB_1$, there exists a smooth map $A:B_1\to GL_{2}(\R)$, and a constant $C>0$ such that on $B_{1}$
\begin{equation}
{ V'}=AV
, \qquad |\det A|\leq C,\qquad \|\dot A\|\leq C.
\end{equation}
where we stress by \eqref{eq:der_along_gamma} that $\dot A=x\partial_y A-y\partial_x A$, hence is a smooth matrix-valued map on $B_{1}$. 

Since $e_1',e_2'$ are smooth near $q$, we can also assume $|g(\nabla_{\dot\gamma} e_2',e_1')|\leq C$, by adapting the constant $C$. Therefore, substituting the expression of $V'$ in terms of $V$ in \eqref{eq:eta_ve_dot_gamma} we find 
\begin{equation}\label{eq:eta_ve_dot_gam_est}
\begin{aligned}
|\eta^\ve(\dot\gamma)|&\leq C+\frac{1}{|X|^2}\left|AV\wedge \frac{d}{dt}\left(AV\right)\right|\\
&\leq C+\frac{1}{|X|^2}\left| AV\wedge\left( \dot A V+A\dot V \right)\right|\\
&\leq C+C^2+ \frac{C}{|X|^2}\left|V\wedge
\dot V \right|.
\end{aligned}
\end{equation}
By definition of $\gamma$ \eqref{eq:def_gamma}, $\dot x=-y$ and $\dot y=x$. Therefore
\begin{equation}
\begin{aligned}
\dot V_1=-\nu y-ya_1(x,y)+xh_1(x,y),\qquad \dot V_2=-ya_0(x,y)+xh_0(x,y),
\end{aligned}
\end{equation}
where $h_1,h_0$ are smooth functions. Collecting terms and recalling $b_{0}(y)=O(y^{k+1})$, we have
\begin{equation}\label{eq:long_det}
\begin{aligned}
V\wedge
\dot V &=\left(x\begin{pmatrix}
\nu+a_1\\ a_0
\end{pmatrix}+y^k\begin{pmatrix}
0\\ \frac{\mu}{k!}+\frac{b_0}{y^k}
\end{pmatrix}\right)\wedge\left(-y\begin{pmatrix}
\nu+a_1\\ a_0
\end{pmatrix}+x\begin{pmatrix}
h_1\\ h_0
\end{pmatrix}\right)\\
&=\ell_1(x,y)x^2+\ell_2(x,y)xy^k+\ell_3(x,y)y^{k+1},
\end{aligned}
\end{equation}
where $\ell_i$, $i=1,2,3$, are smooth functions, hence they are bounded on $B_{1}$, say $|\ell_i|\leq C$ (from this point $C$ denotes a constant suitably chosen). Combining the estimate described in item $(ii)$ of \cref{thm:finite_order_intro}, namely $|X|^2\geq C(x^2+y^{2k})$, with \eqref{eq:long_det}, one has
\begin{equation}
\frac{1}{|X|^2}\left|V\wedge
\dot V \right|\leq |\ell_1|+\frac{1}{2}|\ell_2|+|\ell_3|\frac{|y|^{k+1}}{x^2+y^{2k}}\leq 2C+C\frac{|y|^{k+1}}{x^2+y^{2k}}.
\end{equation}
Substituting the latter inequality in \eqref{eq:eta_ve_dot_gam_est} we are left with
\begin{equation}
|\eta^\ve(\dot\gamma)|\leq C+C\frac{|y|^{k+1}}{x^2+y^{2k}}.
\end{equation}
Therefore, we have the following estimate
\begin{equation}
\int_{C_{\delta}}|\eta^\ve|=\int_{0}^{2\pi}|\eta^\ve(\dot\gamma)|dt\leq 2\pi C+C\int_{0}^{2\pi}\frac{\delta^{k+1}|\sin(t)|^{k+1}}{\delta^2\cos(t)^2+\delta^{2k}\sin(t)^{2k}}dt.
\end{equation}
To conclude the proof of the lemma is sufficient to show that
\begin{equation}
\lim_{\delta\to 0}\int_{0}^{2\pi}\frac{\delta^{k+1}|\sin(t)|^{k+1}}{\delta^2\cos(t)^2+\delta^{2k}\sin(t)^{2k}}dt=\lim_{\delta\to 0}4\int_{0}^{\pi/2}\frac{\delta^{k+1}\sin(t)^{k+1}}{\delta^2\cos(t)^2+\delta^{2k}\sin(t)^{2k}}dt<+\infty.
\end{equation}
The latter can be estimated as follows
\begin{equation}
\begin{aligned}
\int_{0}^{\pi/2}\frac{\delta^{k+1}\sin(t)^{k+1}}{\delta^2\cos(t)^2+\delta^{2k}\sin(t)^{2k}}dt&\leq \int_{0}^{\pi/2}\frac{\delta^{k+1}\left(\sin(t)^{k+1}+k\cos(t)^2\sin(t)^{k-1}\right)}{\delta^2\cos(t)^2+\delta^{2k}\sin(t)^{2k}}dt\\
&=\int_0^{\pi/2}\frac{d}{dt}\left(\arctan\left(\frac{\delta^{k}\sin(t)^k}{\delta\cos(t)}\right)\right)dt=\frac{\pi}{2}.
\end{aligned}
\end{equation}
{We now consider $k=1$. We choose smooth coordinates $(x,y)$ centered at $q$ and a smooth oriented
	orthonormal frame $e'_1,e'_2$ for $g^\ve|_S$, defined near $q$, and write
	$X=V'_1e'_1+V'_2e'_2$. Since $k=1$, $D_qX$ is invertible. Thus, there exists $A\in \mathrm{GL}(2,\mathbb R)$ such that 
	\begin{equation}\label{eq:i_primi}
		\begin{pmatrix}
			V_1'\\
			V_2'
		\end{pmatrix}=A\begin{pmatrix}
		x\\ y
	\end{pmatrix}+O(x^2+y^2).
	\end{equation}
	Let $C_\delta$ be the circle parametrized by
	$\gamma_\delta(t)=(\delta\cos t,\delta\sin t)$. By \eqref{eq:i_primi}, there exist constants $c_1,c_2>0$, independent of $\delta$, such that, along $\gamma_\delta$ we have
	\[
	|V'|\geq c_1\delta,
	\qquad
	|\dot V'|\leq c_2\delta,
	\]
	where $V'=(V'_1,V'_2)$ and dots denote differentiation with
	respect to $t$. By smoothness of the frame $e_1',e_2'$, and since $|\dot\gamma_\delta|=O(\delta)$, there exists a constant $c_3>0$ independent of $\delta$ such that 
	\begin{equation}
		|g(\nabla_{\dot\gamma_\delta}e'_2,e'_1)|\leq c_3\delta.
	\end{equation}
	Thus, applying \eqref{eq:eta_ve_dot_gamma} we may estimate
	\begin{equation}
			\begin{aligned}
			|\eta^\ve(\dot\gamma_\delta)|
			\leq
			|g(\nabla_{\dot\gamma_\delta}e'_2,e'_1)|
			+\frac{|V'\wedge\dot V'|}{|X|^2}
			\leq c_3\delta+\frac{|\dot V'|}{|V'|}
			\leq c_3\delta+\frac{c_2}{c_1},
		\end{aligned}
	\end{equation}
	where we used $|X|=|V'|$ and the three inequalities above. Therefore,
	\[
	\int_{C_\delta}|\eta^\ve|
	=\int_0^{2\pi}|\eta^\ve(\dot\gamma_\delta(t))|\,dt
	\leq 2\pi\left(c_3\delta+\frac{c_2}{c_1}\right),
	\]
	which remains bounded as $\delta\to0$.
	Since the circles $C_\delta$ converge uniformly to $q$ and
	have length $O(\delta)$, this concludes the proof for $k=1$.
}

\end{proof}

\begin{remark}
	 There exists a smooth vector field $X\in\mathfrak X(\mathbb R^2)$ with an isolated singularity at the origin such that  
	\begin{equation}
		\lim_{\delta\to 0}\int_{\partial B_\delta(0)}|\eta|=+\infty.
	\end{equation}
	Indeed, consider the following smooth vector field on $\mathbb R^2$ (extended by continuity at the origin)
	\begin{equation}
		X=\exp\left({-\frac{1}{x^2+y^2}}\right)\left(\cos\left(\frac{y}{(x^2+y^2)^2}\right)\partial_x+\sin\left(\frac{y}{(x^2+y^2)^2}\right)\partial_y\right).
	\end{equation}
	Such vector field has an isolated singularity at the origin. We compute $\eta$ with respect to the Euclidean metric. Let us define a function $\phi$ by imposing $X=|X|\cos\phi\partial_x+|X|\sin\phi\partial_y$. Introducing polar coordinates, $(x,y)=(r\cos\theta,r\sin\theta)$, we note that $|X|=\exp({-}r^{-2})$ and $\phi=r^{-3}\sin\theta$. We compute
	{\begin{equation}
		\begin{aligned}
		\eta(\partial_\theta)&=\frac{1}{|X|^2}g(\nabla_{\partial_\theta}JX, X)=
		g\left(-\partial_\theta\sin\phi\partial_x + \partial_\theta\cos\phi\partial_y, \cos\phi\partial_x+\sin\phi\partial_y\right)=-\partial_\theta\phi=-r^{-3}\cos\theta.
	\end{aligned}
	\end{equation}}
	Denoting with $B_\delta(0)$ the Euclidean ball of radius $\delta$ centred at the origin, we find 
	\begin{equation}
		\lim_{\delta\to 0}\int_{\partial B_\delta(0)}|\eta|=\lim_{\delta\to 0}\int_0^{2\pi}|\eta(\partial_\theta)|d\theta=\lim_{\delta\to 0}\int_{0}^{2\pi}\delta^{-3}|\cos\theta|{ d\theta}=\lim_{\delta\to 0}4\delta^{-3}=+\infty.
	\end{equation}

\end{remark}

            \section{Convergence of measures. Proof of Theorem~\ref{thm:gauss_bonn_lim_intro0}}\label{sec:proof_of_main}

                Let $(S,g)$ be a smooth, compact and orientable Riemannian surface. Recall that $C^1(S)$ is a Banach space, endowed with the norm 
                \begin{equation}
                    \|\varphi \|_{C^1(S)}=\|\varphi\|_{\infty}+\|d\varphi\|_{\infty}.
                \end{equation}
                Its dual space, $C^1(S)^*$, is also a Banach space endowed with the norm 
                \begin{equation}
                    \|\mu\|_{C^1(S)^*}=\sup_{\varphi\neq 0}\frac{|\mu(\varphi)|}{\|\varphi\|_{C^1(S)}}.
                \end{equation}
                We recall that a sequence $\{\mu_\ve\}_{\ve>0}\subset C^1(S)^*$ strongly converges to $\mu_0\in C^1(S)^*$ if 
                \begin{equation}
                    \lim_{\ve\to 0}\|\mu_\ve-\mu_0\|_{C^1(S)^*}=0.
                \end{equation}
  
      \subsection{Proof of Theorem \ref{thm:gauss_bonn_lim_intro0}}
        
          Our goal is to prove the following limit in $C^1(S)^*$
                  \begin{equation}\label{eq:goal1}
                \lim_{\ve\to 0}\left(K^\ve\sigma^\ve-\frac{\mu_{-1}}{\sqrt{\ve}}\right)=2\pi\sum_{i=1}^\ell\mathrm{ind}(q_i)\delta_{q_i},
         \end{equation}
         where, denoting  $\alpha$ the 1-form  in Lemma~\ref{lem:levi_ve}, one has
          \begin{equation}\label{eq:goal2}
                \mu_{-1}=\lim_{\ve\to 0}\sqrt\ve K^\ve\sigma^\ve=d\alpha.
            \end{equation}
            Recall that characteristic points with finite order of degeneracy are isolated. By compactness of $S$ the characteristic set is finite $\Sigma(S)=\{q_1,\dots,q_\ell\}$.
            According to \cref{lem:technical_eta_ve}, for each $q_i$ there exists a family of circles $\{C_\delta(q_i)\}_{\delta >0}\subset S$ enclosing $q_i$ and converging uniformly to it for $\delta\to 0$ such that 
            \begin{equation}
                \lim_{\delta\to 0}\mathrm{length}\left(C_\delta(q_i)\right)=0,\qquad  \limsup_{\delta\to 0}\int_{C_{\delta}(q_i)}|\eta^\ve|=L_{\ve}<+\infty,\qquad \forall \,\ve\in (0,1],
            \end{equation}
            Let $B_\delta(q_i)$ be the connected component of $S\setminus C_{\delta}(q_i)$ containing $q_i$, so that $C_{\delta}(q_i)=\partial B_{\delta}(q_i)$ and let 
            \begin{equation}\label{eq:tanteballs}
                B_\delta:=B_\delta(q_1)\cup\dots \cup B_\delta(q_{\ell}),\qquad S_\delta:=S\setminus B_\delta.
            \end{equation}
            {All boundaries $\partial B_\delta(q_i)$ are oriented according to the convention of \cref{rmk:orientation}, consequently
            	\[
            	\partial S_\delta=-\partial B_\delta
            	=\bigcup_{i=1}^{\ell}\left(-\partial B_\delta(q_i)\right).
            	\]}To prove \eqref{eq:goal1}, we split the proof into more steps.
            
            {\bf Step 1}. We prove the following equality: for $\varphi\in C^1(S)$ and every $\ve\in(0,1]$ one has
             \begin{equation}\label{eq:third_delta_division}
                \int_S\varphi\left(K^\ve\sigma^\ve-\frac{d\alpha}{\sqrt{\ve}}\right)=-\int_{S}d\varphi\wedge \left(\eta^\ve-\frac{\alpha}{\sqrt\ve}
                \right)+2\pi\sum_{i=1}^{\ell}\varphi(q_i)\,\mathrm{ind}(q_i).
            \end{equation}
\begin{proof}[Proof of Step 1.] {We fix $\ve\in(0,1]$}. Corollary~\ref{c:garofalo} implies that $|X|^{-1}$ is Lebesgue integrable on $S$. Then, item {$(ii)$} of \cref{lem:levi_ve} implies that $\|d\alpha\|$ is Lebesgue integrable as well, therefore, for $\varphi\in C^1(S)$ we have
            \begin{equation}
                \int_S\varphi\left(K^\ve\sigma^\ve-\frac{d\alpha}{\sqrt{\ve}}\right)=\lim_{\delta\to 0}\int_{S_\delta}\varphi \left(K^\ve\sigma^\ve-\frac{d\alpha}{\sqrt{\ve}}\right).
            \end{equation}
            Furthermore, writing $K^{\ve}\sigma^{\ve}=d\eta^{\ve}$ by \eqref{eq:levi_cartan} and using Stokes theorem {and $\partial S_\delta=-\partial B_\delta$}, we get
            \begin{equation}\label{eq:first_delta_deivision}
                \begin{aligned}
    \int_S\varphi\left(K^\ve\sigma^\ve-\frac{d\alpha}{\sqrt{\ve}}\right)&=\lim_{\delta\to 0}\int_{S_\delta}\varphi \, d\left(\eta^\ve-\frac{\alpha}{\sqrt\ve}
                \right)\\
                &=\lim_{\delta\to 0}\left(-\int_{S_\delta}d\varphi\wedge \left(\eta^\ve-\frac{\alpha}{\sqrt\ve}
                \right){-}\int_{\partial B_\delta} \varphi\left(\eta^\ve-\frac{\alpha}{\sqrt\ve}
                \right)\right).
                \end{aligned}
            \end{equation}
            According to \eqref{eq:dom_conv}, there exists a constant $C>0$ independent of $\ve$ such that $\|\eta^\ve-\frac{\alpha}{\sqrt{\ve}}\|\leq C/|X|$. Since $|X|^{-1}$ is integrable, we can pass to the limit in the first term in \eqref{eq:first_delta_deivision} and write
             \begin{equation}\label{eq:second_delta_deivision}
                \begin{aligned}
    \int_S\varphi\left(K^\ve\sigma^\ve-\frac{d\alpha}{\sqrt{\ve}}\right)=-\int_{S}d\varphi\wedge \left(\eta^\ve-\frac{\alpha}{\sqrt\ve}
                \right){-}\lim_{\delta\to 0}\int_{\partial B_\delta} \varphi\left(\eta^\ve-\frac{\alpha}{\sqrt\ve}
                \right).
                \end{aligned}
            \end{equation}
            We claim that the following equality holds for the second term in the right hand side (recall that $B_{\delta}$ is a finite union of balls, cf.\ \eqref{eq:tanteballs})
            \begin{equation}\label{eq:any_ve_delta}
                \lim_{\delta\to 0}\int_{\partial B_\delta} \varphi\left(\eta^\ve-\frac{\alpha}{\sqrt\ve}
                \right)={-}\sum_{i=1}^{\ell}2\pi\varphi(q_i)\,\mathrm{ind}(q_i),
            \end{equation}
            To prove the claim, we fix a characteristic point $q_i$ and the corresponding ball $B_{\delta}(q_{i})$ around it {(we stress that $\ve\in (0,1]$ is now fixed)}.  Exploiting property \eqref{eq:levi_int_ind} for the Levi-Civita connection form $\eta^{\ve}$, we can write
            \begin{equation}
            \begin{aligned}
              L:=\lim_{\delta\to 0}\left|\int_{\partial B_\delta(q_i)} \varphi\left(\eta^\ve-\frac{\alpha}{\sqrt\ve}\right){+}2\pi \varphi(q_i)\mathrm{ind}(q_i)\right|
              =\lim_{\delta\to 0}\left|\int_{\partial B_\delta(q_i)} \varphi\left(\eta^\ve-\frac{\alpha}{\sqrt\ve}\right)-\int_{\partial B_\delta(q_i)} \varphi(q_i)\eta^\ve\right|
                  \end{aligned}
            \end{equation}
            so that we can estimate
               \begin{align} 
           L   \leq&\lim_{\delta\to 0}\left|\int_{\partial B_\delta(q_i)} (\varphi-\varphi(q_i))\eta^\ve\right|+\lim_{\delta\to 0}\left|\int_{\partial B_\delta(q_i)}\varphi\frac{\alpha}{\sqrt\ve}\right|\\
                 \leq &\lim_{\delta\to 0}\left(\max_{p\in \partial B_\delta(q_i)}|\varphi(p)-\varphi(q_i)|\right)\limsup_{\delta\to 0}\left(\int_{\partial B_\delta(q_i)}|\eta^\ve|\right)  \label{eq:secondariga}\\
                 &\qquad +\left(\sup_{p\in S\setminus\Sigma(S)}\frac{|\varphi(p)|\|\alpha_p\|}{\sqrt{\ve}}\right)\lim_{\delta\to 0}\left(\mathrm{length}(\partial B_\delta(q_i))\right). \label{eq:terzariga}
            \end{align}
            The term in \eqref{eq:secondariga} is equal to zero, since $\varphi$ is continuous and the lim\,sup is finite by \cref{lem:technical_eta_ve}. Also \eqref{eq:terzariga} vanishes, since the length of $\partial B_\delta$ converges to zero by \cref{lem:technical_eta_ve} and $\|\alpha\|$ is bounded by item {$(i)$} of \cref{lem:levi_ve}. This proves \eqref{eq:any_ve_delta}. Replacing \eqref{eq:any_ve_delta}  into \eqref{eq:second_delta_deivision} we deduce \eqref{eq:third_delta_division} and Step 1 is proved.
           \end{proof}
           
           {\bf Step 2.} We prove the following equality
              \begin{equation}\label{eq:step2}
                \lim_{\ve\to 0}\left\|K^\ve\sigma^\ve-\frac{d\alpha}{\sqrt{\ve}}-2\pi\sum_{i=1}^\ell\mathrm{ind}(q_i)\delta_{q_i}\right\|_{C^1(S)^*}=0.
            \end{equation}

   \begin{proof}[Proof of Step 2.]            Thanks to \eqref{eq:third_delta_division}, for any $\varphi\in C^1(S)$ and {$\ve\in (0,1]$} we have the following inequality 
            \begin{equation}
            \begin{aligned}
            \left|\int_S\varphi\left(K^\ve\sigma^\ve-\frac{d\alpha}{\sqrt{\ve}}\right)-2\pi\sum_{i=1}^\ell \varphi(q_i)\,\mathrm{ind}(q_i)\right|&=\left|\int_{S}d\varphi\wedge \left(\eta^\ve-\frac{\alpha}{\sqrt\ve}
                \right)\right|\\
                &\leq \max_{p\in S}\|d_p\varphi\|\int_S\left\|\eta^\ve-\frac{\alpha}{\sqrt{\ve}}\right\|\sigma^1\leq \|\varphi\|_{C^1(S)}\int_S\left\|\eta^\ve-\frac{\alpha}{\sqrt{\ve}}\right\|\sigma^1,
                \end{aligned}
                \end{equation}
                where in the second line we have used  \eqref{eq:determinant_inequality}.
                By definition of the $C^1(S)^*$ norm we have 
                \begin{equation}\label{eq:norm_sequence_ve}
                \left\|K^\ve\sigma^\ve-\frac{d\alpha}{\sqrt{\ve}}-2\pi\sum_{i=1}^\ell\mathrm{ind}(q_i)\delta_{q_i}\right\|_{C^1(S)^*}\leq \int_S\left\|\eta^\ve-\frac{\alpha}{\sqrt{\ve}}\right\|\sigma^1.
            \end{equation}
            On the one hand, according to \eqref{eq:lim_diff=zero},  the difference $\eta^\ve-\alpha/\sqrt{\ve}$ converges to zero pointwise outside of the characteristic set (which has measure zero).  
            
            On the other hand, $\|\eta^\ve-\alpha/\sqrt{\ve}\|$ by \eqref{eq:dom_conv} is dominated by $|X|^{-1}$, which is integrable. We can therefore apply Lebesgue dominated convergence and taking the limit of $\ve\to 0$ on both sides we prove \eqref{eq:step2}, concluding Step 2. 
         \end{proof}
        
        Step 2 proves \eqref{eq:goal1}. Then \eqref{eq:goal2} is a direct consequence, since
            \begin{equation}
               \left\|\sqrt\ve K^\ve\sigma^\ve-d\alpha\right\|_{C^1(S)^*}\leq \sqrt\ve\left\| K^\ve\sigma^\ve-\frac{d\alpha}{\sqrt\ve}-2\pi\sum_{i=1}^\ell \mathrm{ind}(q_i)\delta_{q_i}\right\|_{C^1(S)^*}+\sqrt\ve\left\|2\pi\sum_{i=1}^\ell\mathrm{ind}(q_i)\delta_{q_i}\right\|_{C^1(S)^*}.
            \end{equation}
            and taking $\ve\to 0$ of both sides we get \eqref{eq:goal2}, concluding the proof of Theorem \ref{thm:gauss_bonn_lim_intro0}.            

\begin{remark}\label{rmk:mu_-1}
In the proof of \cref{thm:gauss_bonn_lim_intro0} the measure $\mu_{-1}$ is explicitly computed and coincides with the exterior differential of the form $\alpha$ defined in \eqref{eq:alpha}. In particular 
\begin{equation}
	\mu_{-1}=-d\left(\frac{\mathrm{div}(X)}{|X|}\omega|_S\right),
\end{equation} 
where $\omega|_S$ is the contact form normalized as in \eqref{eq:omega_norm} and restricted to $S$.
\end{remark}

\section{An operative formula for the invariants}
	
    We provide an operative formula to compute the order of degeneracy of a characteristic point and its associated invariants, through the introduction of an auxiliary affine connection. 
    
    Let us first prove this preliminary lemma. Throughout the section when referring to $\nabla$ as the Levi-Civita connection on $TS$, we mean with respect to $g^1|_S$.
    
     \begin{lemma}\label{prop:formula_lambda0}
            	Let $\nabla$ be the Levi-Civita connection on $TS$. There exists a neighborhood $U\subset S$ of $q$ such that the following linear endomorphism is (fibrewise) diagonalizable
            	\begin{equation}
            		\nabla X:TS|_U\to TS|_U,\qquad Y\mapsto \nabla_YX.
            	\end{equation}
            	In particular $\nabla X$ admits an eigenvector field $v_0\in\mathfrak X(U)$ satisfying $v_0|_q\in\ker D_q X$ and $|v_0|_q=1$. {If $q$ has finite even order of degeneracy, we choose $v_0|_q$ to be positively oriented according to \cref{rmk:Lambda_sign}.} 
            \end{lemma}
   \begin{proof}

Since $X_q=0$, the equality $\nabla X|_q=D_qX$ holds (see \cref{r:tresei}). According to \cref{r:tresei}, $\mathrm{div}_q (X)=\mathrm{trace}\,(D_qX)\neq 0$. Moreover $\det (D_qX)=0$. We deduce that the characteristic polynomial of $\nabla X|_q$ has the following expression
\begin{equation}
	P(\lambda)|_q=\lambda(\lambda-\mathrm{div}_q(X)).
\end{equation}
Therefore $\nabla X|_q$ has two distinct eigenvalues at $q$. By continuity $\nabla X$ has two distinct eigenvalues in a neighbourhood $U\subset S$ of $q$. In particular these are smooth functions $\lambda_0,\lambda_1:U\to \mathbb R$ satisfying 
\begin{equation}\label{eq:lambda_0(q)=0}
	\lambda_0(q)=0,\qquad\lambda_1(q)=\mathrm{div}_q(X).
\end{equation}
Then there exists an oriented couple  $(v_0,v_1)$ of normalized eigenvectors
\begin{equation}\label{eq:eigen_vec/val}
	\nabla_{v_i}X=\lambda_i v_i,\qquad |v_i|_q=1,\qquad i=0,1,
\end{equation}
Since $\lambda_{0}(q)=0$, the vector field $v_{0}$ satisfies the requirements.
   \end{proof}         
     In the following proposition, whose proof is postponed to Section~\ref{app:formula_lambda}, we provide a formula for the invariants  $\Lambda^{(k)}(q)$ defined in \eqref{eq:k-th_eigenvalue_intro}.
            \begin{prop}\label{prop:formula_lambda}
            	Let $q\in S$ be a characteristic point with $\mathrm{ord}(q)=k\geq 2$. Let $\nabla$ be the Levi-Civita connection on $TS$. Let $v_0\in\mathfrak X(U)$ be the vector field of Lemma~\ref{prop:formula_lambda0}. Then
            	\begin{equation}\label{eq:lambdagrande1}
            		\Lambda^{(k)}(q)=\left.v_0^{(k-1)}\left(\frac{\det(\nabla X)}{\mathrm{trace}(\nabla X)}\right)\right|_q.
            	\end{equation}
            	Furthermore, $k$ is the smallest natural number such that the right-hand side of \eqref{eq:lambdagrande1} is non zero.
            \end{prop}
Given  $\nabla$ the Levi-Civita connection on $TS$ we define the function
\begin{equation}\label{eq:hat_K_body}
	\widehat K:S\to\mathbb R,\qquad \widehat K(p)=-1+\frac{\det (\nabla X|_{p})}{\mathrm{trace}^2(\nabla X|_{p})}.
\end{equation}
\begin{prop}
	Let $q\in S$ be a characteristic point with $\mathrm{ord}(q)=k$ odd, and let $v_0$ be the local vector field introduced in  Lemma~\ref{prop:formula_lambda0}  and $\widehat K$ be the function \eqref{eq:hat_K_body}. Then 
	\begin{equation}
		\mathrm{div}_q(X)\Lambda^{(k)}(q)=v_0^{(k-1)}\left(\widehat{K}+1\right)\Big|_{q},
	\end{equation}
	 If $S$ is compact and all characteristic points are of finite order of degeneracy, denoting $k_i=\mathrm{ord}(q_{i})$ for $i=1,\ldots,\ell$ one has
	\begin{equation}\label{eq:secondpart}
		\chi(S)=\sum_{k_i\,\,\text{odd}}\mathrm{sign}\left(v_0^{(k_i-1)}\left(\widehat{K}+1\right)\Big|_{q_{i}}\right).
	\end{equation} 
\end{prop} 
\begin{proof}
Let us start by writing
\begin{equation}
	\begin{aligned}
		v_0^{(k-1)}\left(\widehat{K}+1\right)=v_0^{(k-1)}\left(\frac{\det\nabla X}{(\mathrm{trace}\nabla X)^2}\right)=v_0^{(k-1)}\left(\frac{\det\nabla X}{\mathrm{trace}\nabla X}\cdot \frac{1}{\mathrm{trace}\nabla X}\right)
			\end{aligned}
\end{equation}
Since $k$ is the smallest natural number such that the right-hand side of \eqref{eq:lambdagrande1} is non zero one has
\begin{equation}
	\begin{aligned}
		v_0^{(k-1)}\left(\widehat{K}+1\right)\Big|_{q}=v_0^{(k-1)}\left(\frac{\det\nabla X}{\mathrm{trace}\nabla X}\right)\bigg|_{q} \frac{1}{\mathrm{trace}(\nabla X|_{q})}=v_0^{(k-1)}\left(\frac{\det\nabla X}{\mathrm{trace}\nabla X}\right)\bigg|_{q} \mathrm{trace}(\nabla X|_{q})
			\end{aligned}
\end{equation}
where, in the last equality, we used in a crucial way that $\mathrm{div}_q(X)=\mathrm{trace}(\nabla X|_q)=\pm 1$ at characteristic points thanks to
 \cref{prop:div=pm} and \cref{r:tresei}.

Equality \eqref{eq:secondpart} follows from Poincaré-Hopf theorem combined with $(iii)$ of \cref{thm:finite_order_intro}.
\end{proof}

\subsection{Proof of \cref{prop:formula_lambda}}\label{app:formula_lambda}

Recall that there exists a couple  $(v_0,v_1)$ of normalized eigenvectors
\begin{equation}\label{eq:eigen_vec/val}
	\nabla_{v_i}X=\lambda_i v_i,\qquad |v_i|_q=1,\qquad i=0,1,
\end{equation}
whose associated eigenvalues satisfy, at characteristic points, 
 \begin{equation}\label{eq:lambda_0(q)=0}
	\lambda_0(q)=0,\qquad\lambda_1(q)=\mathrm{div}_q(X).
\end{equation}
 Let us fix an arbitrary affine connection $\nabla'$ (possibly with torsion). Let $ v_i'$ and $ \lambda_i'$, for $i=0,1$, be the eigenvectors and eigenvalues satisfying \eqref{eq:eigen_vec/val} associated with $\nabla'$ (notice that the arguments of \cref{r:tresei}, and consequently those of \cref{prop:formula_lambda0}, apply to arbitrary affine connections). Let $\mathcal C$ (resp.~$\mathcal C'$) be the integral curve of $v_0$ (resp.~$v'_{0}$) passing through $q$. {When $\mathrm{ord}(q)=k$ is even we choose the vector $v_0|_q,v_0'|_q\in\ker D_qX$ to be positively oriented according to \cref{rmk:Lambda_sign} (in particular, since the kernel is one-dimensional, $v_0|_q=v_0'|_q$).} We divide the proof into more steps.

\textbf{Step 1.} We have the identity:  $\mathrm{order}(\lambda_{0}|_{\mathcal C},q)=\mathrm{order}({\lambda}_{0}'|_{{\mathcal C}'},q)$.

\noindent
 We start by proving the existence of two smooth 1-forms $\beta_0,\beta_1$ defined near $q$, such that 
\begin{equation}\label{eq:beta_diff_forms}
	\lambda_0'=\lambda_0+\beta_0(X),\qquad v_0'=v_0+\beta_1(X)v_1.
\end{equation}
Let $A$ be the following tensor defined by the difference of the connections
\begin{equation}
	A(Y,Z)=\nabla_{Y}'Z-\nabla_{Y}Z,\qquad \forall\,Y,Z\in\mathfrak X(S).
\end{equation}
Let us decompose $A$ with respect to the basis $v_0,v_1$, defining the smooth bilinear forms $A^{0},A^{1}$ 
\begin{equation}
	A=A^0v_0+A^1v_1.
\end{equation}
By construction $v_0'|_q$ and $v_0|_q$ have norm $1$ and both belong to $\ker D_qX$, which is 1-dimensional. Hence $v_0'|_q=\pm v_0|_q$. Up to multiplying $v_0'$ by a smooth non vanishing function, we may assume 
\begin{equation}
	v_0'=v_0+fv_1.
\end{equation}
for some smooth function $f$ vanishing at $q$.
We  compute  $\nabla'_{v'_{0}}X$ in two ways: on the one hand
\begin{equation}
	\begin{aligned}
	\nabla'_{v_0'}X&=\nabla_{v_0'}X+A(v_0',X)=\nabla_{v_0+fv_1}X+A(v_0',X)\\
	&=\lambda_0 v_0+f\lambda_1v_1+A(v_0',X)=(\lambda_0+ A^0(v_0',X))v_0+(f\lambda_1+A^1(v_0',X))v_1.
	\end{aligned}
\end{equation}
On the other hand  by definition $\nabla'_{v'_{0}}X=\lambda'_{0}v'_{0}=\lambda_0'(v_0+fv_1)$, hence one can solve for $f$ finding
\begin{equation}
\lambda_0'=\lambda_0+A^0(v_0',X),\qquad  v_0'=v_0+\frac{A^1(v_0',X)}{\lambda_0'-\lambda_1}v_1.
\end{equation} 
Setting $\beta_0(\cdot)=A^0(v_0',\cdot)$ and $\beta_1(\cdot)=A^1(v_0',\cdot)/(\lambda_0'-\lambda_1)$, equation \eqref{eq:beta_diff_forms} is proved.

Let ${\mathcal C}'$ be the integral curve of  $v_0'$ through $q$. Let $\mathrm{order}(\mathcal C\cap{\mathcal C}',q)$ denote the order of contact  between $\mathcal C$ and $\mathcal C'$at $q$  (see \cref{def:order_of_contact}). Then \cref{lem:order_orbits} together with the expression of $v_0'$ in \eqref{eq:beta_diff_forms} ensure 
\begin{equation}\label{eq:O1}
	\mathrm{order}(\mathcal C\cap{\mathcal C}',q)\geq\mathrm{order}(\beta_1(X)|_{\mathcal C},q).
\end{equation}
Moreover, due to  $\nabla_{v_0}X=\lambda_{0}v_{0}$, the following holds
\begin{equation}\label{eq:nabla_vanishing_order}
	\nabla^{(i)}_{v_0}X|_{q}=0,\qquad 0\leq i\leq \mathrm{order}(\lambda_0|_{\mathcal C},q),
\end{equation}
which, in turn, thanks to $\nabla_{v_{0}}\beta_{i}(X)=(\nabla_{v_{0}}\beta_{i})X+\beta_{i}(\nabla_{v_{0}}X)$ implies that 
\begin{equation}\label{eq:O2}
	\mathrm{order}(\beta_i(X)|_{\mathcal C},q)>\mathrm{order}(\lambda_0|_{\mathcal C},q),\qquad i=0,1.
\end{equation}
Combining \eqref{eq:beta_diff_forms} and \eqref{eq:O2} one gets 
\begin{equation}\label{eq:order=order'}
\mathrm{order}(\lambda_{0}|_{\mathcal C},q)=\mathrm{order}(\lambda_0'|_{\mathcal C},q).
\end{equation}
Consequently, we deduce that
\begin{equation}\label{eq:some_est_on_order}
	\mathrm{order}(\mathcal C\cap {\mathcal C}',q)\geq\mathrm{order}(\beta_1(X)|_{\mathcal C},q)>\mathrm{order}(\lambda_{0}|_{\mathcal C},q)=\mathrm{order}(\lambda_0'|_{\mathcal C},q).
\end{equation}
We conclude by \cref{lem:order_lem1} the desired identity
\begin{equation}\label{eq:all_orders_are_same}
	\mathrm{order}(\lambda_{0}|_{\mathcal C},q)=\mathrm{order}({\lambda}'_{0}|_{\mathcal C},q)=\mathrm{order}({\lambda}_{0}'|_{{\mathcal C}'},q).
\end{equation}

\textbf{Step 2.} The identity \eqref{eq:lambdagrande1} hold for a particular connection $\nabla'$ with $k=\mathrm{ord}(q)=\mathrm{order}(\lambda'_0|_{\mathcal C'}, q)+1$.

\noindent
According to \cref{lem:normal_form_strong_body} there exist coordinates $(x,y)$ near $q$ such that 
\begin{equation}
	q=(0,0),\qquad X=x(\nu+a_1(x,y))\partial_x+\left(\frac{\mu}{k!} y^k+xa_0(x,y)+b_0(y)\right)\partial_y,
\end{equation}
Moreover, in these coordinates, the curve $y\mapsto (0,y)$ is an horizontal kernel extension and $\Lambda^{(k)}(q)=\mu$. Fix now a Riemannian metric on $S$ which, near $q$, has the expression $dx^2+dy^2$. Let $\nabla'$ be the associated Levi-Civita connection. The following equality holds 
\begin{equation}\label{eq:stanc}
	\left.\left(\nabla'_{\partial_y}X\right)\right|_{(0,y)}=\left(\frac{\mu}{(k-1)!}y^{k-1}+\partial_{y}b_0(y)\right)\partial_y.
\end{equation}
Since the right hand side of \eqref{eq:stanc} is proportional to $\partial_{y}$, we deduce that 
\begin{equation}
	v_{0}'|_{(0,y)}={\varphi(y)}\partial_y,\qquad \lambda'_0|_{(0,y)}=\mu\frac{y^{k-1}}{(k-1)!}+\partial_{y}b_0(y),
\end{equation}
where ${\varphi}$ is a smooth function satisfying ${\varphi(0)=\pm 1}$ {(since the curve $y\mapsto (0,y)$ is an horizontal kernel extension, when $k$ is even we actually have $\varphi(0)=1$, in accordance with our choice of orientation for $v_0$ at the beginning of the proof and \cref{rmk:Lambda_sign})}. It follows that $y\mapsto (0,y)$ is a re-parametrization of the integral curve $\mathcal C'$ of $v_0'$ passing through the origin. Therefore, by property \eqref{eq:needed_orders_app_body} of $b_0$ $$\mathrm{order}(\lambda'_0|_{\mathcal C'}, q)=k-1=\mathrm{ord}(q)-1.$$ Applying \eqref{eq:reparametrization} we obtain
\begin{equation}
	\left.v_0'^{(k-1)}\left(\frac{\det\nabla' X}{\mathrm{trace}\nabla' X}\right)\right|_q={(\varphi(0))^{k-1}}\left.\frac{d^{k-1}}{dy^{k-1}}\left(\left.\frac{\det\nabla' X}{\mathrm{trace}\nabla' X}\right|_{(0,y)}\right)\right|_{y=0}=\left.\frac{d^{k-1}}{dy^{k-1}}\left(\left.\frac{\det\nabla' X}{\mathrm{trace}\nabla' X}\right|_{(0,y)}\right)\right|_{y=0}
\end{equation}
{since $\varphi(0)^{k-1}=1$. Indeed, for odd $k$ this follows from $\varphi(0)=\pm1$, while for even $k$ we have $\varphi(0)=1$.} By property \eqref{eq:needed_orders_app_body} of $b_0$, $k$ is the smallest natural number such that $\frac{d^{k-1}}{dy^{k-1}}\lambda_0'(0,y)|_{y=0}$ is non zero, hence
\begin{equation}
	\begin{aligned}
		\left.\frac{d^{k-1}}{dy^{k-1}}\left(\left.\frac{\det\nabla' X}{\mathrm{trace}\nabla' X}\right|_{(0,y)}\right)\right|_{y=0}&=
		\left.\frac{d^{k-1}}{dy^{k-1}}\left(\left.\lambda'_0\cdot\frac{\lambda'_1}{\lambda'_1+\lambda'_0}\right|_{(0,y)}\right)\right|_{y=0}\\
		&=\left.\frac{d^{k-1}}{dy^{k-1}}\lambda'_0(0,y)\right|_{y=0}\frac{\lambda'_1(0,0)}{\lambda'_1(0,0)+\lambda'_0(0,0)}\\
		&=\left.\frac{d^{k-1}}{dy^{k-1}}\lambda'_0(0,y)\right|_{y=0}=\mu=\Lambda^{(k)}(q).
	\end{aligned}
\end{equation}

{\textbf{Step 3.}} Equation \eqref{eq:lambdagrande1} holds for the Levi-Civita connection.

\noindent
{Combining Steps 1 and 2} we see that, for any connection $k-1=\mathrm{ord}(q)-1=\mathrm{order}(\lambda_0|_{\mathcal C},q)$. Thus, $k$ is the smallest natural number such that $v_0^{(k-1)}(\lambda_0)|_q$ is non zero and one has
\begin{equation}
	v_0^{(k-1)}\left(\frac{\det(\nabla X)}{\mathrm{trace}(\nabla X)}\right)\bigg|_q
	=v_0^{(k-1)}\left(\lambda_0\cdot\frac{\lambda_1}{\lambda_1+\lambda_0}\right)\bigg|_q
	=v_0^{(k-1)}\left(\lambda_0\right)\bigg|_q\frac{\lambda_1(q)}{\lambda_1(q)+\lambda_0(q)}=v_0^{(k-1)}\left(\lambda_0\right)\bigg|_q,
\end{equation}
where we have used $\lambda_1(q)=\mathrm{div_q(X)}\neq 0$, and $\lambda_0(q)=0$. On the other hand, by \eqref{eq:beta_diff_forms} and \eqref{eq:O2}
\begin{equation}
	v_0^{(k-1)}\left(\lambda_0\right)\bigg|_q=v_0^{(k-1)}\left(\lambda_0'-\beta_0(X)\right)\bigg|_q=v_0^{(k-1)}\left(\lambda_0'\right)\bigg|_q.
\end{equation}
Furthermore, \eqref{eq:some_est_on_order} and the fact that $v_0|_q=\pm v_0'|_q$ allow applying \cref{lem:order_lem1} and \cref{r:lemmachiave} to the curves $\mathcal C$, $\mathcal C'$ and the function $\lambda_0'$, finding 
\begin{equation}
	v_0^{(k-1)}\left(\lambda_0'\right)\bigg|_q=(\pm1)^{k-1}v_0'^{(k-1)}\left(\lambda_0'\right)\bigg|_q=v_0'^{(k-1)}\left(\lambda_0'\right)\bigg|_q=v_0'^{(k-1)}\left(\frac{\det(\nabla' X)}{\mathrm{trace}(\nabla' X)}\right)\bigg|_q,
\end{equation}
{where we have used $(\pm1)^{k-1}=1$: for odd $k$ this
follows from parity of $k-1$, while for even $k$ the orientation convention
ensures that $v_0|_q=v_0'|_q$.}
Choosing $\nabla'$ as in {Step 2} and combining the latter three equations, we prove the claim.

We conclude the proof: by {Step 3} \eqref{eq:lambdagrande1} holds. Combining {Step 1 and 2} we see that $k$ is the smallest natural number such that the right hand side of \eqref{eq:lambdagrande1} is non zero.\qed
{	\begin{remark}\label{rmk:infinite_order_formula}
		For every degenerate characteristic point $q$, with $v_0$
		as in \cref{prop:formula_lambda0}, we have
		\begin{equation}\label{eq:ord_form}
			\mathrm{ord}(q)
		=
		\inf\left\{
		k\in\mathbb N,\ k\geq1:
		\left.v_0^{(k-1)}
		\left(\frac{\det(\nabla X)}{\mathrm{trace}(\nabla X)}\right)
		\right|_q\neq0
		\right\},
		\end{equation}
		with the convention $\inf\varnothing=+\infty$. To see this, choose a horizontal kernel extension $\gamma$ and smooth
		coordinates such that $\gamma(t)=(0,t)$. By definition of $\lambda_\gamma$ in \cref{prop:critical_curve_lambda_intro} we have $X|_{(0,y)}=\lambda_\gamma(y)\partial_y$. Therefore, for the flat
		connection $\nabla'$ in these coordinates we have
		\[
		\left.\nabla'_{\partial_y}X\right|_{(0,y)}
		=\frac{d\lambda_\gamma}{dy}(y)\partial_y.
		\]
		Along $\mathcal C':=\gamma((-\ve,\ve))$, we may choose
		$v'_0=\partial_y$, and we have
		$\lambda'_0(0,y)=\frac{d\lambda_\gamma}{dy}(y)$. Consequently
		\[
		\mathrm{ord}(q)
		=1+\mathrm{order}(\lambda'_0|_{\mathcal C'},q),
		\]
		also for infinite orders. Moreover, for every connection,
		\[
		\frac{\det(\nabla X)}{\mathrm{trace}(\nabla X)}
		=\lambda_0\frac{\lambda_1}{\lambda_0+\lambda_1},
		\]
		where the second factor $\lambda_1/(\lambda_0+\lambda_1)$ is smooth and nonzero near $q$.
		Thus \eqref{eq:ord_form} holds for the flat connection $\nabla'$. Finally, the order comparison in Step~1 of
		Section~\ref{app:formula_lambda} also applies to infinite orders:
		if either $\mathrm{order}(\lambda'_0|_{\mathcal C'},q)$
		or $\mathrm{order}(\lambda_0|_{\mathcal C},q)$ is finite,
		the argument in Step 1 forces the other to be equal to it.
		Otherwise both are infinite. 
		Therefore
		\[
		\mathrm{order}(\lambda_0|_{\mathcal C},q)
		=
		\mathrm{order}(\lambda'_0|_{\mathcal C'},q),
		\]
		where $\mathcal C$ is the integral curve of $v_0$ through $q$.
		This proves the formula \eqref{eq:ord_form} for every connection and in particular for the Levi-Civita connection.
	\end{remark}
}

\subsection{An observation in the analytic case} As an application of \cref{prop:v_0_intro}, we prove that in the analytic setting an isolated characteristic point necessarily has finite order of degeneracy. 
	\begin{prop}\label{prop:curve_of_char}
		Let $S$ be an analytic surface embedded in an analytic 3D contact sub-Riemannian manifold. Let $q\in S$ be a characteristic point with $\mathrm{ord}(q)=+\infty$, then $q$ is contained in an embedded analytic curve of characteristic points. In particular $q$ is not isolated.
	\end{prop}
	\begin{proof}
		 Since data are analytic, the Levi-Civita connection $\nabla$ of $g^1|_S$ is also analytic. Consequently the eigenvector field $v_0$ of \cref{prop:formula_lambda0} and the corresponding simple eigenvalue $\lambda_0$ are  analytic as well. Since $\mathrm{ord}(q)=+\infty$, {\cref{rmk:infinite_order_formula}} implies that 
		\begin{equation}
		v_0^{(k)}\left(\frac{\det(\nabla X)}{\mathrm{trace}(\nabla X)}\right)\bigg|_q=0,\qquad\forall\,k\in\mathbb N.
		\end{equation}
		Repeating the computation of the previous section we find for every $k\in\mathbb N$ (recall $\lambda_0(q)=0$)
		\begin{equation}
			v_0^{(k)}\left(\frac{\det(\nabla X)}{\mathrm{trace}(\nabla X)}\right)\bigg|_q
			=v_0^{(k)}\left(\lambda_0\cdot\frac{\lambda_1}{\lambda_1+\lambda_0}\right)\bigg|_q
			=v_0^{(k)}\left(\lambda_0\right)\bigg|_q\frac{\lambda_1(q)}{\lambda_1(q)+\lambda_0(q)}=\left.v_0^{(k)}\left(\lambda_0\right)\right|_q.
		\end{equation}
		Combining the two latter equations $v_0^{(k)}(\lambda_0)|_q=0$ for all $k\in\mathbb N$. Let $\gamma:(-\ve,\ve)\to M$ be an embedded integral curve of $v_0$ passing through $q$ at $t=0$, which is analytic. One has
		\begin{equation}
			\lambda_0(\gamma(t))= 0,\qquad \forall\,t\in(-\ve,\ve).
		\end{equation}
		By definition of $\gamma$, it holds $\dot\gamma(t)=v_0|_{\gamma(t)}$. We compute 
		\begin{equation}
			\nabla_{\dot\gamma(t)}X=\nabla_{v_0}X|_{\gamma(t)}=\lambda_0(\gamma(t)){v_0}|_{\gamma(t)}=0,\qquad \forall\,t\in(-\ve,\ve).
		\end{equation}
		Therefore the characteristic field $X$ is parallel along $\gamma$ and vanishes at $\gamma(0)=q$. It follows that $X|_{\gamma(t)}=0$ for all $t\in (-\ve,\ve)$ and that $\gamma$ is an analytic curve of characteristic points containing $q$.
	\end{proof}	
The proof of \cref{c:garofalo2} now follows easily.
 \begin{proof}[Proof of \cref{c:garofalo2}]
 	 Since $S$ is analytic and characteristic points are assumed to be isolated, by \cref{prop:curve_of_char} every characteristic point has finite order of degeneracy. Hence, according to \cref{c:garofalo}, $|X|^{-1}\in L^1_{loc}(S,\sigma)$ for any smooth measure $\sigma$. The estimate $|\mathcal{H}|\leq C |X|^{-1}$ (cf.\ for instance in  \cite[Section 2.1]{Rossi_deg}) concludes the proof.
 \end{proof}

\appendix

        \section{Some technical lemmas}\label{a:contact}

            \subsection{Orders of contact}\label{ssec:ord_cont}
            We recall the definition of order of vanishing of a function and order of contact of two curves, which will be needed in what follows.

            	Let $S$ be a smooth surface and $\mathcal C\subset S$ be an embedded curve. Let $q\in \mathcal C$ and let $f:S\to \mathbb R$ a smooth function. Let $\gamma:\mathbb R\to \mathcal C$ be a smooth  regular parametrization such that $\gamma(0)=q$, then the order of $f$ at $q$, along $\mathcal C$, is defined as
            	\begin{equation}
            		\mathrm{order}(f|_{\mathcal C},q)=\inf\left\{i\in\mathbb N\,:\,\left.\frac{d^i}{dt^{i}}f(\gamma(t))\right|_{t=0}\neq 0\right\}.
            	\end{equation}
 
            It is elementary to see that the order of a function at a point along a curve does not depend on the parametrization since for $f:\mathbb R\to\mathbb R$ a smooth function having order of vanishing $k\in\mathbb N$ at the origin, and $\varphi:\mathbb R\to\mathbb R$ a diffeomorphism such that $\varphi(0)=0$, then 
            	\begin{equation}\label{eq:reparametrization}
            		\left.\frac{d^k}{dt^k}f(\varphi(t))\right|_{t=0}=\left(\varphi'(0)\right)^kf^{(k)}(0).
            	\end{equation}
Let us define the order of contact of two embedded curves.
        	\begin{defi}\label{def:order_of_contact}
        		Let $S$ be a smooth surface and $\mathcal C_1,\mathcal C_2\subset S$ be two embedded curves. Then $\mathcal C_1,\mathcal C_2$ are said to have contact of order $k\in\mathbb N\cup\{+\infty\}$ at $q\in \mathcal C_1\cap\mathcal C_2$, and we write 
        		\begin{equation}
        			\mathrm{order}(\mathcal C_1\cap\mathcal C_2,q)=k,
        		\end{equation}
        		if $k\in \N$ is the largest integer such that there exist smooth regular parametrizations $\gamma_i:\mathbb R\to\mathcal C_i$, for  $i=1,2$, such that 
        		\begin{equation}
        			\gamma_1(0)=\gamma_2(0)=q,\qquad \frac{d^j\gamma_1}{dt^j}(0)=\frac{d^j\gamma_2}{dt^j}(0),\qquad \forall\,\,0\leq j\leq k.
        		\end{equation}
        	\end{defi}
        	If the order of contact of the intersection of two curves is large enough, one can use equivalently any of the curve to compute the order of vanishing of a function.
            \begin{lemma}\label{lem:order_lem1}
            	Let $S$ be a smooth surface, $f:S\to \mathbb R$ be a smooth function. Let $\mathcal C_1,\mathcal C_2\subset S$ be smooth embedded curves and assume that
            	\begin{equation}
            		\mathrm{order}(\mathcal C_1\cap\mathcal C_2,q)\geq \mathrm{order}(f|_{\mathcal C_1},q).
            	\end{equation}
            	Then
            	\begin{equation}
            		\mathrm{order}(f|_{\mathcal C_1},q)=\mathrm{order}(f|_{\mathcal C_2},q).
            	\end{equation}
            \end{lemma}
            
        \begin{proof} 
        	Let $m=\mathrm{order}(\mathcal C_1\cap\mathcal C_2,q)$ and $k=\mathrm{order}(f|_{\mathcal C_1},q)$. By definition of order of contact of curves there exist smooth regular parametrizations $\gamma_i:\mathbb R\to S$, $i=1,2$, such that for every smooth  $f:S\to \mathbb R$ 
        	\begin{equation}
        		\left.\frac{d^i}{dt^i}f(\gamma_1(t))\right|_{t=0}=\left.\frac{d^i}{dt^i}f(\gamma_2(t))\right|_{t=0},\qquad \forall\,0\leq i\leq m.
        	\end{equation}
        	Since by hypothesis $m\geq k$, we can conclude that $\mathrm{order}(f|_{\mathcal C_1},q)=\mathrm{order}(f|_{\mathcal C_2},q)$. 

        \end{proof}
    \begin{remark}  \label{r:lemmachiave}      Notice that 
           if ${\mathrm{order}(\mathcal C_1\cap \mathcal C_2,q)\geq k}:=\mathrm{order}(f|_{\mathcal C_1},q)=\mathrm{order}(f|_{\mathcal C_2},q)$ and $\gamma_i:\mathbb R\to\mathcal C_i$, $i=1,2$, are smooth parametrizations satisfying $\gamma_1(0)=\gamma_2(0)=q$ and $\dot\gamma_1(0)=a\dot\gamma_2(0)$, with $a\neq 0$, then 
            	\begin{equation}  \label{eq:lemmachiave} 
            		\left.\frac{d^k}{dt^{k}}f(\gamma_1(t))\right|_{t=0}=a^k\left.\frac{d^k}{dt^{k}}f(\gamma_2(t))\right|_{t=0}.      
            	\end{equation}
	\end{remark}

  If a curve is a regular level set, we can then relate the two notions just introduced as follows.
            \begin{lemma}\label{lem:order=order_of_eq}
            	Let $S$ be a smooth surface, and $\mathcal C_1,\mathcal C_2\subset S$ be smooth embedded curves. Assume that $u_2:S\to \mathbb R$ is a submersion such that $\mathcal C_2=u_{2}^{-1}(0)$. Then 
            	\begin{equation}
            		\mathrm{order}(\mathcal C_1\cap\mathcal C_2,q)=\mathrm{order}(u_2|_{\mathcal C_1},q)-1.
            	\end{equation}
            \end{lemma}            
        \begin{proof}
        	There exist local coordinates $(x,y)$ near $q$ such that 
  $q=(0,0)$ and $u_2(x,y)=y$.
        	Let $\gamma_1:\mathbb R\to\mathcal C_1$ be a regular parametrization with $\gamma_1(0)=(0,0)$ and let us write
$
        		\gamma_1(t)=(x_1(t),y_1(t))
$.
        	Since $q\in\mathcal C_1\cap\mathcal C_2$, we deduce that $x_1(0)=y_1(0)=0$. 
	
	Let $n=\mathrm{order}(\mathcal C_1\cap\mathcal C_2,(0,0))$. If $n=0$, then necessarily $\dot y_1(0)\neq 0$. On the other hand
        	\begin{equation}
        		\mathrm{order}(y|_{\mathcal C_1},(0,0))=\mathrm{order}(y\circ\gamma_1,0)=\mathrm{order}(y_1,0)=1=\mathrm{order}(\mathcal C_1\cap\mathcal C_2,(0,0))+1.
        	\end{equation}
       If $n\geq 1$, then necessarily $\dot x_1(0)\neq 0$, therefore up to reparametrization we may assume $x_1(t)=t$, or equivalently
       $
        		\gamma_1(t)=(t,y_1(t))
$.
        	Since $\mathcal C_2=\{(x,0)\in\mathbb R^2\mid x\in\mathbb R\}$ we deduce that 
        	\begin{equation}
        		n=\mathrm{order}(y_1,0)-1=\mathrm{order}(y\circ\gamma_1,0)-1=\mathrm{order}(y|_{\mathcal C_1},(0,0))-1. \qedhere
        	\end{equation}
        \end{proof}
        We also list here another useful property of the order of vanishing.
    	\begin{lemma}\label{lem:order_orbits}
    	Let $f:\mathbb R^2\to\mathbb R$ be a smooth function and $X,Y\in\mathfrak X(\mathbb R^2)$ be vector fields, linearly independent at the origin. Let $\mathcal C_1,\mathcal C_2$ be the orbits of the flows of $X$ and $X+fY$ through the origin, where $f\in C^\infty(\mathbb R^2)$, then 
    	\begin{equation}
    	\mathrm{order}(\mathcal C_1\cap\mathcal C_2,(0,0))= \mathrm{order}(f|_{\mathcal C_1},(0,0)).
    	\end{equation}
    	\end{lemma}
   		\begin{proof}
   		If $f(0,0)\neq 0$, since $X,Y$ are linearly independent at the origin one has $$\mathrm{order}(\mathcal C_1\cap\mathcal C_2,(0,0))=0=\mathrm{order}(f|_{\mathcal C_1},(0,0)).$$ Assume now that $f(0,0)=0$. Since $X,Y$ are linearly independent at the origin, there exist coordinates $(x,y)$ such that 
   		\begin{equation}
   			X=\partial_x,\qquad X+fY=\partial_x+f(a\partial_x+b\partial_y),
   		\end{equation}
   		where $a,b$ are some smooth functions and $b(0,0)\neq 0$.
  		Notice that in these coordinates $$\mathcal C_1=\{(x,y)\in\mathbb R^2\,:\, y=0\}.$$From \cref{lem:order=order_of_eq} we can infer that
  		\begin{equation}\label{eq:chain_order_ineq}
  			\begin{aligned}
  				\mathrm{order}(\mathcal C_1\cap\mathcal C_2,(0,0))&=\mathrm{order}(y|_{\mathcal C_2},(0,0))-1=\inf\{k\in\mathbb N\,:\,(\partial_x+fY)^{(k)}(y)|_{(0,0)}\neq 0\}-1\\
  				&= \inf\{k\in\mathbb N,\,{k\geq 1}:\,(\partial_x+fY)^{(k-1)}(\partial_x+fY)(y)|_{(0,0)}\neq 0\}-1\\
  				&=\inf\{k\in\mathbb N,\,{k\geq 1}:\,(\partial_x+fY)^{(k-1)}(fb)|_{(0,0)}\neq 0\}-1.
  			\end{aligned}
  		\end{equation}
  		By induction one can show that, for any $k\in\mathbb N$, there exist smooth functions $a_0,\dots,a_{k-1}$ such that 
  		\begin{equation}
		(\partial_x+fY)^k(fb)=b\partial^k_{x}f+\sum_{j=0}^{k-1}a_j\partial_x^jf.
		\end{equation}
  		Since $b(0,0)\neq 0$, we see that the smallest $k\in\mathbb N$ such that $(\partial_x+fY)^k(fb)|_{(0,0)}\neq 0$ coincides with the smallest $k$ for which $\partial_x^kf(0,0)\neq 0$. In other words 
  		\begin{equation}
  			\begin{aligned}
  				\inf\{k\in\mathbb N,\,{k\geq 1}:\,(\partial_x+fY)^{(k-1)}(fb)|_{(0,0)}\neq 0\}
  				&= \inf\{k\in\mathbb N,\,{k\geq 1}:\,(\partial_x)^{(k-1)}(f)|_{(0,0)}\neq 0\}\\
  				&=\mathrm{order}(f|_{\mathcal C_1},{(0,0)})+1.
  			\end{aligned}
  		\end{equation}
  		The proof can be concluded combining the latter equality with \eqref{eq:chain_order_ineq}.
  	\end{proof}
 
        \subsection{The degree of a class of singularities}\label{app:mild_wt_connection}

\begin{lemma}\label{lem:index_lem}
    Let $\nu,\mu\in\mathbb R\setminus\{0\}$ and let $k\in\mathbb N$ and consider the vector field in $\R^{2}$
    \begin{equation}
        X=\nu x\partial_x+\frac{\mu}{k!} y^k\partial_y.
    \end{equation}
    Then, the index of the origin can be computed as
    \begin{equation}
        \mathrm{ind}(0,0)=\begin{cases}
            0,\qquad &k\,\,\text{even},\\
            \mathrm{sign}(\nu\mu),\qquad &k\,\,\text{odd}.
        \end{cases}
    \end{equation}
\end{lemma}
\noindent\textit{Proof.} The index of the origin for $X$ coincides with the degree of the map
    \begin{equation}
    \varphi:\mathbb S^1\to \mathbb S^1,\qquad \varphi(\cos\theta,\sin\theta)=\frac{(\nu\cos\theta,\frac{\mu}{k!}\sin^k\theta)}{|(\nu\cos\theta,\frac{\mu}{k!}\sin^k\theta)|}.
    \end{equation}
For $k$ even, $\varphi$ is not surjective, hence its image is contractible and its degree is zero.
Now assume $k$ odd. For $s\in[0,1]$ consider the vector field
    \begin{equation}
    X_s=\nu x\partial x+ \frac{\mu}{k!}\left(s y+(1-s)y^k\right)\partial_y,
    \end{equation}
and note that the only vanishing point of $X_s$ is the origin since $k$ odd. Therefore the map
\begin{equation}
F:[0,1]\times \mathbb S^1\to \mathbb S^1,\qquad F(s,(x,y))=\left.\frac{X_s}{|X_s|}\right|_{(x,y)}
\end{equation}
is a smooth homotopy. Moreover, $X_0=X$, thus the degree at the origin of $X$ equals the degree at the origin of $X_1=\nu x\partial_x+\frac{\mu}{k!}y\partial_y$. The latter field has a non degenerate singularity at the origin, therefore
\begin{equation}
\mathrm{ind}(0,0)=\mathrm{sign}\det(D_{(0,0)} X_1)=\mathrm{sign}\left(\det\begin{pmatrix}
\nu & 0 \\ 0 & \frac{\mu}{k!}
\end{pmatrix}\right)=\mathrm{sign}(\nu\mu).\qed
\end{equation} 
{\subsection{Area of sublevel sets of smooth functions}
\begin{lemma}\label{lem:sublevel_nondegeneracy}
	Let $S$ be a smooth closed surface, with a smooth area form $\sigma$. Let $f:S\to \mathbb R$ be a continuous function, which is smooth in a neighborhood of $f^{-1}(0)$. If
	\[
	\Phi(c):=\int_{\{f\leq c\}}\sigma=O(c)
	\qquad\text{as }c\to0^+,
	\]
	then every zero of $f$ is a non-degenerate minimum.
\end{lemma}
\begin{proof}
	Since $\Phi(c)\to0$ as $c\to 0^+$, the function $f$ cannot be negative
	on any nonempty open set. Since $f$ is continuous, we have $f\geq0$.
	Fix $q\in f^{-1}(0)$ and choose coordinates $x$ on a neighborhood $U\simeq \mathbb R^2$ centered
	at $q$. Since $f(q)=0$ and $f\geq 0$, $q$ is necessarily a minimum for $f$. Therefore in coordinates $x$ we get
	\begin{equation}\label{eq:H}
	f(x)=\frac12 x^T Hx+o(|x|^2),
	\end{equation}
	where $H$ is a positive semidefinite symmetric matrix. Moreover, in these coordinates $\sigma=\rho(x)dx$, where $\rho:\mathbb R^2\to \mathbb R$ is a positive smooth function. We may now compute
	\begin{equation}
		\begin{aligned}
			\liminf_{c\to0^+}\frac{\Phi(c)}{c}&\geq\liminf_{c\to0^+}\frac{1}{c}\int_{\{f\leq c\}\cap U}\sigma=\liminf_{c\to0^+}\frac{1}{c}\int_{\R^2}\chi_{\left\{\left(\frac12 x^T Hx+o(|x|^2)\right)\leq c\right\}}\rho(x)dx\\
				&=\liminf_{c\to0^+}\int_{\R^2}\chi_{\left\{\left(\frac12 y^T Hy+\frac{1}{c}o(c|y|^2)\right)\leq 1\right\}}\rho(\sqrt{c}y)dy\geq \rho(0)\int_{\R^2}\chi_{\{\frac12 y^T Hy\leq 1\}}dy,
		\end{aligned}
	\end{equation}
	where we used \eqref{eq:H}, the rescaling $x=\sqrt{c}\,y$, and Fatou's lemma. If $H$ were singular, then the function $\chi_{\{\frac12 y^T Hy\leq 1\}}$ would not be in $L^{1}(\R^{2})$, contradicting $\Phi(c)=O(c)$. Therefore $H$ is positive definite and $q$ is a non-degenerate minimum.
\end{proof}}

             \section{The horizontal plane in the Heisenberg group}\label{a:heis}
        In this section, we explicitly compute the sequence $\{K^{\ve}\sigma^{\ve}\}\subset C^1(S)^*$ for the horizontal plane in the Heisenberg group. The following lemma will be used throughout the section.
         \begin{lemma}\label{lem:delta_conv}
         	Let $\{\psi_\ve\}_{\ve>0}\subset L^1(\mathbb R^n)$ be a sequence of non-negative functions. If, for every $\rho>0$
         	\begin{equation}
	         		(i)\,\lim_{\ve\to 0}\int_{B_\rho(0)}\psi_\ve(x)dx=2\pi,\qquad(ii)\,\lim_{\ve\to 0}\int_{\mathbb R^n\setminus B_\rho(0)}\psi_\ve(x)dx=0,
         	\end{equation}
         then $\psi_\ve\to2\pi\delta_0$ weakly in $C^1_c({\R^n})^*$, i.e., for any $\varphi\in C^1_c({\R^n})$ it holds 
         \begin{equation}
         	\lim_{\ve\to 0}\int_{\mathbb R^n}\varphi(x)\psi_\ve(x)dx=2\pi\varphi(0).
         \end{equation}
         \end{lemma}
     	\begin{proof}
     	 Let $\varphi\in C^1_c(\mathbb R^n)$. Combining $(i)$ and $(ii)$ one finds
     	\begin{equation}
     		\begin{aligned}
     			\lim_{\ve\to 0}\left|\int_{\mathbb R^n}\varphi(x)\psi_\ve(x)dx-2\pi\varphi(0)\right|&=\left|\lim_{\ve\to 0}\int_{\mathbb R^n}\varphi(x)\psi_\ve(x)dx-2\pi\varphi(0)\right|\\
     			&=\left|\lim_{\ve\to 0}\int_{\mathbb R^n}\varphi(x)\psi_\ve(x)dx-\lim_{\ve\to 0}\int_{\mathbb R^n}\varphi(0)\psi_\ve(x)dx\right|\\
     			&=\lim_{\ve\to 0}\left|\int_{\mathbb R^n}(\varphi(x)-\varphi(0))\psi_\ve(x)dx\right|.
     		\end{aligned}
     	\end{equation}
     	Therefore, for every $\rho>0$ we have 
     	\begin{equation}
     		\begin{aligned}
     	\lim_{\ve\to 0}\left|\int_{\mathbb R^n}\varphi(x)\psi_\ve(x)dx-2\pi\varphi(0)\right|&=\lim_{\ve\to 0}\left|\int_{\mathbb R^n}(\varphi(x)-\varphi(0))\psi_\ve(x)dx\right|\leq\lim_{\ve\to 0} \int_{\mathbb R^n}|\varphi(x)-\varphi(0)|\psi_\ve(x)dx\\
     	&=\lim_{\ve\to 0} \int_{B_\rho(0)}|\varphi(x)-\varphi(0)|\psi_\ve(x)dx+\lim_{\ve\to 0} \int_{\mathbb R^n\setminus B_\rho(0)}|\varphi(x)-\varphi(0)|\psi_\ve(x)dx\\
     	&\leq 2\pi\|\varphi-\varphi(0)\|_{L^\infty(B_\rho(0))}.
     		\end{aligned}
     	\end{equation}
     Since $\rho>0$ is arbitrary and $\varphi$ is smooth, taking the limit of $\rho\to 0$ of both sides of the latter inequality yields the result.
     	\end{proof}
         \subsection{The plane $z=0$ in the Heisenberg group}
        Let $(\mathbb R^3,\mathcal D,g)$ be the Heisenberg group: the contact distribution $\mathcal D$ is expressed as the kernel of the normalized contact form $\omega$ which, in cylindrical coordinates $x=r\cos\theta, y=r\sin\theta, z$, reads 
        \begin{equation}\label{eq:example_form_metrics0}
        	\omega=dz+\frac{r^2}{2}d\theta.
        \end{equation}
        The extended metrics \eqref{eq:g_ve_properties0} have the expression 
        \begin{equation}
        	g^\ve=dr^2+r^2d\theta^2+\frac{1}{\ve}\omega\otimes \omega.
        \end{equation}
        Let $S=\{(x,y,z)\mid z=0\}\subset \R^{3}$. Substituting the expression of $\omega$ into the one of $g^\ve$ and imposing $dz=0$ we obtain
        \begin{equation}
        	g^\ve|_S=dr^2+f_\ve(r)^2d\theta^2,\qquad\text{where}\qquad  f_\ve(r)=r\sqrt{1+\frac{r^2}{4\ve}}.
        \end{equation}
     	A standard computation shows that 
     	\begin{equation}
     		K^\ve\sigma^\ve=-f''_\ve(r)drd\theta=-\frac{r \left(r^2 + 6\varepsilon\right)}{\sqrt{\varepsilon} \left(r^2 + 4\varepsilon\right)^{3/2}}drd\theta.
     	\end{equation}
     	The form $d\alpha$ in the statement of \cref{thm:gauss_bonn_lim_intro0} can be computed as
     	\begin{equation}
     		d\alpha=-\left(\lim_{\ve\to0}\sqrt{\ve}\frac{r \left(r^2 + 6\varepsilon\right)}{\sqrt{\varepsilon} \left(r^2 + 4\varepsilon\right)^{3/2}}\right)drd\theta=-drd\theta.
     	\end{equation}
\begin{lemma}
	For every $\varphi\in C^1_{c}(\mathbb R^2)$, the following equality holds 
	\begin{equation}
		\lim_{\ve\to 0}\int_{\mathbb R^2}\varphi\left(K^\ve\sigma^\ve-\frac{\mu_{-1}}{\sqrt\ve}\right)
		=2\pi\varphi(0).
	\end{equation}
\end{lemma}
\begin{proof}
	Based on the computations above, we can write 
\begin{equation}
	K^\ve\sigma^\ve-\frac{d\alpha}{\sqrt\ve}=\left(\frac{1}{\sqrt{\ve}}-\frac{r \left(r^2 + 6\varepsilon\right)}{\sqrt{\varepsilon} \left(r^2 + 4\varepsilon\right)^{3/2}}\right)drd\theta=\left(\frac{1}{r\sqrt{\ve}}-\frac{ \left(r^2 + 6\varepsilon\right)}{\sqrt{\varepsilon} \left(r^2 + 4\varepsilon\right)^{3/2}}\right)dxdy.
\end{equation}
It is sufficient to show that the function 
\begin{equation}
	\psi_\ve=\frac{1}{r\sqrt{\ve}}-\frac{ \left(r^2 + 6\varepsilon\right)}{\sqrt{\varepsilon} \left(r^2 + 4\varepsilon\right)^{3/2}},
\end{equation}
 	satisfies the hypothesis of \cref{lem:delta_conv}. Notice that $\psi_\ve>0$. It remains to show $(i)$ and $(ii)$.
 	
 	\noindent
 	$(i)$  Let $\rho>0$, then
 	\begin{equation}\label{eq:cond_1}
 		\begin{aligned}
 			\int_{B_\rho(0)}\psi_\ve(x,y)dxdy&=2\pi \int_0^{\rho}\left(-f_\ve''(r)+\frac{1}{\sqrt\ve}\right)dr=2\pi\left.\left(-f'_\ve(r)+\frac{r}{\sqrt\ve}\right)\right|_{0}^{\rho}\\
 			&=2\pi\left(-\frac{2 \ve + \rho^{2}}{\sqrt{\ve} \sqrt{4 \ve + \rho^{2}}}+\frac{\rho}{\sqrt\ve}\right)+2\pi,
 		\end{aligned}
 	\end{equation}
 	where we have used the explicit expression of $f'_\ve$
 	\begin{eqnarray}
 		f'_\ve(r)=\frac{2 \ve + r^{2}}{\sqrt{\ve} \sqrt{4 \ve + r^{2}}}.
 	\end{eqnarray}
 	Taking the limit of $\ve\to 0$ of both sides in \eqref{eq:cond_1} gives 
 	\begin{equation}
 		\begin{aligned}
 				\lim_{\ve\to 0}\int_{B_\rho(0)}\psi_\ve(x,y)dxdy&=2\pi+2\pi\lim_{\ve\to 0}\frac{\rho\sqrt{4 \ve + \rho^{2}}-\rho^{2}-2\ve}{\sqrt{\ve} \sqrt{4 \ve + \rho^{2}}}\\
 				&=2\pi+2\pi\lim_{\ve\to 0}\frac{\rho^2\sqrt{1+\frac{4 \ve}{\rho^2}}-\rho^{2}-2\ve}{\sqrt{\ve} \sqrt{4 \ve + \rho^{2}}}\\
 				&=2\pi+2\pi\lim_{\ve\to 0}\frac{\rho^2+2\ve+o(\ve)-\rho^2-2\ve}{\sqrt{\ve} \sqrt{4 \ve + \rho^{2}}}=2\pi,
 		\end{aligned}
 	\end{equation}
 	which is condition $(i)$ of \cref{lem:delta_conv}.
 	
 	\noindent
 	$(ii)$  Fix $\rho>0$, then, for $x\in \mathbb R^n\setminus B_\rho(0)$ we have 
 	\begin{equation}
 		\begin{aligned}
 			\int_{\mathbb R^2\setminus {B_\rho(0)}}|\psi_\ve(x,y)|dxdy=\int_{\mathbb R^2\setminus B_\rho(0)}\psi_\ve(x,y)dxdy&=2\pi \int_\rho^{+\infty}\left(-f_\ve''(r)+\frac{1}{\sqrt\ve}\right)dr=\frac{{2\pi}(2 \ve +{\rho}^{2})}{\sqrt{\ve} \sqrt{4 \ve + \rho^{2}}}-\frac{{2\pi}\rho}{\sqrt\ve}.
 		\end{aligned}
 	\end{equation}
 Repeating the computation in the proof of point $(i)$ and taking the limit as $\ve\to 0 $ yields point $(ii)$.
 \end{proof}

    {\bf Conflict of interests statement.}  The authors declare that there is no conflict of interest regarding the publication of this paper.

    {\bf Data availability statement.}  No data were used to support this study.

\bibliographystyle{alphaabbr}
	\bibliography{bibliography11}

\end{document}